\documentclass[12pt,reqno]{amsart}
\usepackage{amsmath,amsfonts,amsthm,amssymb,color}

\usepackage[T1]{fontenc}
\usepackage{mathrsfs}

\newcommand{\txx}{\textcolor}

\definecolor{dg}{rgb}{0, 0.5, 0}
\definecolor{dp}{rgb}{0.50, 0, 0.40}

\newcommand{\bbd}{{\mathbb{D}}}

\newcommand{\lp}{\left(}
\newcommand{\rp}{\right)}
\newcommand{\lc}{\left[}
\newcommand{\rc}{\right]}

\newcommand{\be}{\mathbf{E}}

\newcommand{\1}{{\bf 1}}

\newcommand{\beq}{\begin{equation}}
\newcommand{\eeq}{\end{equation}}
\newcommand{\beqs}{\begin{equation}}
\newcommand{\eeqs}{\end{equation}}
\newcommand{\bea}{\begin{eqnarray}}
\newcommand{\eea}{\end{eqnarray}}
\newcommand{\beas}{\begin{eqnarray*}}
\newcommand{\eeas}{\end{eqnarray*}}

\newcommand{\cf}{{\mathcal F}}

\newcommand{\PP}{{\mathbb P}}

\newcommand{\R}{{\mathbb R}}

\newtheorem{theorem}{Theorem}[section]

\newtheorem{lemma}[theorem]{Lemma}

\newtheorem{proposition}[theorem]{Proposition}
\theoremstyle{remark}
\newtheorem{remark}[theorem]{Remark}
\theoremstyle{remark}
\newtheorem{example}[theorem]{Example}
\theoremstyle{remark}

\newtheorem{foo}[theorem]{Remarks}

\title[Fractional stochastic integral Volterra equations]{Existence and  smoothness of the density of the solution  to  fractional stochastic integral Volterra equations}

\author[M. Besal\'u, D. M\'arquez-Carreras and E. Nualart]{Mireia Besal\'u, David M\'arquez-Carreras and Eulalia Nualart}

\date{\today}

\address{Mireia Besal\'u, Dep. Gen\`etica, Microbiologia i Estad\'istica, Universitat de Barcelona. Diagonal, 645, 08028 Barcelona}
\email{mbesalu@ub.edu}
\address{David  M\'arquez-Carreras, Facultat de Matem\`atiques i Inform\`atica, Universitat de Barcelona. Gran Via de les Corts Catalanes, 585, 08007 Barcelona}
\email{davidmarquez@ub.edu}
\address{Eulalia Nualart, Universitat Pompeu Fabra and Barcelona Graduate School of Economics, Ram\'on Trias Fargas 25-27, 08005
Barcelona, Spain.}
\email{eulalia.nualart@upf.edu}

\thanks{M. Besal\'u and D. M\'arquez-Carreras are supported by the grant MTM 2015-65092-P from MINECO, Spain. E. Nualart is supported from the Spanish Government grants
PGC2018-101643-B-I00, SEV-2015-0563, and Ayudas Fundaci\'on BBVA a Equipos de Investigaci\'on Cientifica 2017}

\begin{document}

\begin{abstract}
We consider stochastic Volterra integral equations driven by a fractional Brownian motion with Hurst parameter $H>\frac12$. We first derive supremum norm estimates for the solution and its Malliavin derivative. We then show existence and smoothness of the density under suitable nondegeneracy conditions. This extends the results in \cite{HN} and \cite{NS} where stochastic differential equations driven by fractional Brownian motion are considered. The proof uses a priori estimates for deterministic differential equations driven by a function in a suitable Sobolev space.
\end{abstract}

\maketitle

\renewcommand{\theequation}{1.\arabic{equation}}
\setcounter{equation}{0}
\section {Introduction}

We consider the stochastic integral Volterra equation on $\R^d$
\begin{equation}
X_t= X_0+\int_0^t b(t,s,X_s)ds+\int_0^t\sigma(t,s,X_s)dW_s^H,\quad t\in (0,T],  \label{eq:Volterra}
\end{equation}
where
$\sigma=(\sigma^{i,j})_{d \times m}:[0,T]^2 \times \R^d  \to \R^d \times \R^m$ and $b=(b^{i})_{d \times 1}:[0,T]^2 \times \R^d \to \R^d$ are measurable functions,
$W^H=\{W^{H,j}_t, t \in [0,T], j=1,\ldots, m\}$ are independent fractional Brownian motions (fBm) with Hurst parameter $H>\frac12$ defined in a complete probability space $(\Omega, \cf, \PP)$, and $X_0$ is a d-dimensional random variable. 

As $H>\frac12$, the integral with respect to $W^H$ can  be defined as a pathwise Riemann-Stieltjes integral
using the results by Young \cite{Y}.
Moreover, Z\"alhe \cite{Z} introduced a generalized Stieltjes integral using the techniques of fractional calculus. In particular, she obtained a formula for the Riemann-Stieltjes integral  using fractional derivatives
(see (\ref{eq:forpart}) below). Using this formula, Nualart and Rascanu \cite{N-R} proved a general result on existence, uniqueness and finite moments of the solution to a class of general differential equations included in \eqref{eq:Volterra}.
These results were extended by Besal\'u and Rovira \cite{BR} for the Volterra equation 
(\ref{eq:Volterra}). The proof of these results uses a priori estimates for a deterministic differential equation driven by a function in a suitable Sobolev space.

The first aim of this paper it is to obtain supremum norm estimates of the solution to (\ref{eq:Volterra}). We first consider the case where $\sigma$ is bounded since, in this case, the estimates are of polynomial type, while in the general case are of exponential type. In the case where $\sigma$ is bounded, we also obtain estimates for the Malliavin derivative of the solution and show existence and smoothness of the density. To obtain these  results, we first derive a priori estimates for some deterministic equations. Finally, in the case where $\sigma$ is not necessarily bounded, we also show existence of the density by first showing the Fr\'echet differentiability of the solution to the corresponding deterministic equation. 

These results provide extensions of the works by Hu and Nualart \cite{HN} and Nualart and Saussereau \cite{NS}, where stochastic differential equations driven by fBm  are considered. In particular, we provide a corrected proof of \cite[Theorem 7]{HN}, as there is a problem in their argument. The techniques used to obtain the a priori estimates in the present paper are much more involved than those in \cite{HN} and  \cite{NS} due to the time-dependence of the coefficients. As in those papers, our nondegeneracy assumption is an ellipticity-type condition, see Baudoin and Hairer \cite{BH} for the  existence and smoothness of the density under H\"ormander's condition for stochastic differential equations driven by a fBm with Hurst parameter $H>\frac12$.

Volterra equations driven by general It\^o processes or semimartingales are widely  studied,
see for instance \cite{AN,BM1,BM2,P}. Concerning
Volterra equations driven by fBm, the main references are the papers of Deya and Tindel \cite{DT1,DT2},
where existence and uniqueness is studied separately for the case $H >\frac13$ and 
 $H>\frac12$, using an algebraic integration setting and the Young integral, respectively.
 For the case $H>1/2$ and using the Young integral, existence and uniqueness of the solution  to equation (\ref{eq:Volterra}) with an  extra term driven by an independent Wiener process is proved in
 \cite{WY}. 
 See also \cite{S09, D15} for the existence and  uniqueness of fBm  driven Volterra equations in a Hilbert space. In \cite{ZY}, a class of fractional stochastic Volterra equations of convolution type driven by infinite dimensional fBm with Hurst index $H \in (0,1)$ is considered, and existence and regularity results of the stochastic convolution process are established. Last  but not least, existence of the density of the solution to equation (\ref{eq:Volterra}) in the  one dimensional case is obtained in
 \cite{F15} as a consequence of a Bismut type formula.
However, supremum norm estimates and existence and smoothness of the density in the multidimensional case do not seem to be studied yet in the literature for this kind of equations.

The structure of this paper is as follows: in the next section we introduce all the spaces, norms  and operators used through the paper. In Section 3, we obtain a priori  estimates for the solution of some systems of equations in a deterministic framework and study the Fr\'echet differentiability of one of them. Section 4 is devoted to apply the results obtained in Section 3 to the Volterra equation (\ref{eq:Volterra}) and derive the existence and smoothness of the density.

\vskip 12pt
{\bf Notation:} For any integer $k \geq 1$, we denote by $C^k_b$ the class of real-valued functions on $\R^d$ which are $k$ times continuously differentiable with bounded partial derivatives up to the $k$th order.
We denote by $C^{\infty}_b$ the the class of real-valued functions on $\R^d$ which are  infinitely differentiable and bounded together with all their derivatives.

Throughout all the paper, $C_{\alpha}$, $C_{\alpha, \beta}$, $c_{\alpha,T}$, etc. will denote generic constants that may change from line to line.
\renewcommand{\theequation}{2.\arabic{equation}}
\setcounter{equation}{0}
\section{Preliminaries}

For any $\alpha \in (0, \frac12)$, we denote by $W^{\alpha}_1(0,T;\R^d)$ the space of measurable functions $f:[0,T] \rightarrow \R^d$ such that
$$
\Vert f \Vert_{\alpha,1}:=\sup_{t \in [0,T]} \left( \vert f(t) \vert
+\int_0^t \frac{\vert f(t)-f(s) \vert}{\vert t-s \vert^{\alpha+1}} ds\right) < \infty.
$$
For any $\alpha \in (0, \frac12)$, we denote by $W^{1-\alpha}_2(0,T;\R^m)$ the space of measurable functions $g:[0,T] \rightarrow \R^m$ such that
$$
\Vert g \Vert_{1-\alpha,2}:=\sup_{0\leq s < t \leq T} \left( \frac{\vert g(t)-g(s) \vert}{\vert t-s \vert^{1-\alpha}}
+\int_s^t \frac{\vert g(y)-g(s) \vert}{\vert y-s \vert^{2-\alpha}} dy\right) < \infty.
$$
For any $0<\lambda \leq 1$, and any interval $[a,b ]\subset [0,T]$, we denote by $C^{\lambda}(a,b;\R^d)$ the space of $\lambda$-H\"older continuous functions $f:[a,b] \rightarrow \R^d$ equipped with the norm
$$
\Vert f \Vert_{a,b,\lambda}:=\Vert f \Vert_{a,b,\infty}+
\sup_{a \leq s <t \leq b} \frac{\vert f(t)-f(s) \vert}{\vert t-s\vert ^{\lambda}}
$$
where $\Vert f \Vert_{a,b,\infty}:=\sup_{t \in [a,b]} \vert f(t) \vert$.
We set $\Vert f \Vert_{\lambda}=\Vert f \Vert_{0,T,\lambda}$ and
$\Vert f \Vert_{\infty}=\Vert f \Vert_{0,T,\infty}$.

Clearly, for any $\epsilon>0$,
\begin{equation}\label{inclusion}
C^{1-\alpha+\epsilon}(0,T;\R^m) \subset W^{1-\alpha}_2(0,T;\R^m)
\subset C^{1-\alpha}(0,T;\R^m).
\end{equation}
Moreover, as $\alpha \in (0, \frac12)$, 
$$
C^{1-\alpha}(0,T;\R^m) \subset W^{\alpha}_1(0,T;\R^m).
$$
For $d=m=1$, we simply write $W^{\alpha}_1(0,T)$, $W^{1-\alpha}_2(0,T)$, and $C^{\lambda}(0,T)$.

If $f\in C^{\lambda }(a,b)$ and $g\in C^{\mu }(a,b)$ with $\lambda
+\mu >1$, it is proved in \cite{Z} that the Riemman-Stieltjes integral $%
\int_{a}^{b}fdg$ exists and it can be expressed as 
\begin{equation}
\int_{a}^{b}fdg=(-1)^{\alpha }\int_{a}^{b}D_{a+}^{\alpha
}f(t)D_{b-}^{1-\alpha }g_{b-}(t)dt,  \label{eq:forpart}
\end{equation}
where $g_{b-}(t)=g(t)-g(b), $ $1-\mu <\alpha <\lambda $, and the fractional derivatives are defined as
\begin{eqnarray*}
D_{a+}^{\alpha }f(t) &=&\frac{1}{\Gamma (1-\alpha )}\left( \frac{f(t)}{%
(t-a)^{\alpha }}+\alpha \int_{a}^{t}\frac{f(t)-f(s)}{(t-s)^{\alpha +1}}%
ds\right) ,  \label{d2} \\
D_{b-}^{\alpha }f(t) &=&\frac{(-1)^{\alpha }}{\Gamma (1-\alpha )}\left( 
\frac{f(t)}{(b-t)^{\alpha }}+\alpha \int_{t}^{b}\frac{f(t)-f(s)}{%
(s-t)^{\alpha +1}}ds\right).  \label{d3}
\end{eqnarray*}
We refer to \cite{N-R} and \cite{Z} and the references therein for a detailed account about this generalized integral and the fractional calculus.

Let $\Omega=C_0([0,T]; \R^m)$ be the Banach space of continuous functions, null at time 0, equipped with the supremum norm.
Let ${\rm P}$ be the unique probability measure on $\Omega$ such that the canonical process $\{W^H_t, t \in [0,T]\}$ is an $m$-dimensional fractional Brownian motion with Hurst parameter $H>\frac12$.

We denote by $\mathcal{E}$ the space of step functions on $[0,T]$ with values in $\R^m$. Let $\mathcal{H}$ be the Hilbert space defined as the closure of $\mathcal{E}$ with respect to the scalar product
$$
\langle ({\bf 1}_{[0, t_1]},\ldots,{\bf 1}_{[0, t_m]} ),({\bf 1}_{[0, s_1]},\ldots,{\bf 1}_{[0, s_m]} ) \rangle_{\mathcal{H}}=\sum_{i=1}^m R_H(t_i, s_i),
$$ 
where
$$
R_H(t, s)=\int_0^{t \wedge s} K_H(t,r) K_H(s,r) dr,
$$
and $K_H(t,s)$ is the square integrable kernel defined by
\begin{equation} \label{kh}
K_H(t,s)=c_H s^{1/2-H} \int_s^t (u-s)^{H-3/2} u^{H-1/2} du,
\end{equation}
where $c_H=\sqrt{\frac{H(2H-1)}{\beta(2-2H, H-1/2)}}$, $\beta$ denotes the Beta function and $t>s$. For $t \leq s$, we set $K_H(t,s)=0$.

The mapping $({\bf 1}_{[0, t_1]},\ldots,{\bf 1}_{[0, t_m]} ) \rightarrow \sum_{i=1}^m W^{H,i}_{t_i}$ can be extended to an isometry between $\mathcal{H}$ and the Gaussian space $\mathcal{H}_1$ associated to $W^H$. We denote this isometry by $\varphi \rightarrow W^H(\varphi)$.

Consider the operator $K^{\ast}_H$ from $\mathcal{E}$ to $L^2(0,T;\R^m)$
defined by
$$
(K^{\ast}_H \varphi)^i (s)=\int_s^T \varphi^i(t) \partial_t K_H(t,s) dt.
$$
From (\ref{kh}), we get
$$
\partial_t K_H(t,s)=c_H \left(\frac{t}{s} \right)^{H-1/2} (t-s)^{H-3/2}.
$$
Notice that
$$
K^{\ast}_H({\bf 1}_{[0, t_1]},\ldots,{\bf 1}_{[0, t_m]} )
=(K_H(t_1,\cdot), \ldots, K_H(t_m,\cdot)).
$$
For any $\varphi, \psi \in \mathcal{E}$, 
$$
\langle \varphi, \psi  \rangle_{\mathcal{H}}=\langle K^{\ast}_H \varphi, K^{\ast}_H \psi \rangle_{L^2(0,T;\R^m)}={\rm E}(W^H(\varphi)W^H(\psi))
$$
and $K^{\ast}_H$ provides an isometry between the Hilbert space $\mathcal{H}$ and a closed subspace of $L^2(0,T;\R^m)$.

Following \cite{NS}, we consider the fractional version of the Cameron-Martin space $\mathcal{H}_H:=\mathcal{K}_H(L^2(0,T;\R^m))$, where for $h \in   L^2(0,T;\R^m)$,
$$
(\mathcal{K}_H h)(t):=\int_0^t K_H(t,s) h_s ds.
$$
We finally denote by $\mathcal{R}_H=\mathcal{K}_H \circ \mathcal{K}_H^{\ast}: \mathcal{H} \rightarrow   \mathcal{H}_H$ the operator
$$
\mathcal{R}_H \varphi=\int_0^{\cdot} K_H(\cdot, s) (\mathcal{K}^{\ast}_H h)(s) ds.
$$
We remark that for any $\varphi \in \mathcal{H}$, $\mathcal{R}_H \varphi$ is H\"older continuous of order $H$. Therefore, for any $1-H<\alpha<1/2$,
$$
\mathcal{H}_H \subset C^{H}(0,T;\R^m) \subset W^{1-\alpha}_2(0,T;\R^m).
$$
Notice that $\mathcal{R}_H {\bf 1}_{[0,t]}=R_H(t, \cdot)$, and, as a consequence, $\mathcal{H}_H$ is the Reproducing Kernel Hilbert Space associated with the Gaussian process $W^H$. The injection
$\mathcal{R}_H: \mathcal{H} \rightarrow \Omega$ embeds $\mathcal{H}$ densely into $\Omega$ and for any $\varphi \in \Omega^{\ast} \subset  \mathcal{H}$,
$$
{\rm E} \left( e^{i W^H(\varphi)}\right)=
\exp \left( -\frac12 \Vert \varphi \Vert^2_{\mathcal{H}} \right).
$$
As a consequence, $(\Omega, \mathcal{H}, {\rm P})$ is an abstract Wiener space in the sense of Gross.

\renewcommand{\theequation}{3.\arabic{equation}}
\setcounter{equation}{0}
\section{Deterministic differential equations}

Fix $0<\alpha<\frac12$. Consider the deterministic differential equation on $\R^d$
\begin{equation}\label{eq:det}
x_t=x_0+\int_0^t b(t,s,x_s)ds+\int_0^t \sigma(t,s,x_s)dg_s,\qquad t\in[0,T],
\end{equation}
where $g \in W_2^{1-\alpha}(0,T; \R^m)$, $x_0 \in \R^d$, and $b$ and $\sigma$ are as in (\ref{eq:Volterra}).
\vskip 12pt
Consider the following hypotheses on $b$ and $\sigma$:
\begin{itemize}
\item[\bfseries(H1)] $\sigma:[0,T]^2\times\R^d\rightarrow \R^d\times\R^m$ is a measurable function such that the derivatives  $\partial_x \sigma(t,s,x)$,  $\partial_t \sigma(t,s,x)$ and $\partial^2_{x,t} \sigma(t,s,x)$ exist. Moreover,
there exist some constants  $0<\beta,\, \mu, \, \delta\leq
1$ and for every $N\geq 0$ there exists $K_N>0$ such that the following properties hold:

\begin{enumerate}
\item $\left|\sigma(t,s,x)-\sigma(t,s,y)\right| + \left|\partial_{t}\sigma(t,s,x)-\partial_{t}\sigma(t,s,y)\right|\leq K\left|x-y\right|,$\\$\forall x,y\in\R^d,\;\forall s,t\in[0,T]$,
\item $\left|\partial_{x_i}\sigma(t,s,x)-\partial_{y_i}\sigma(t,s,y)\right|+\left|\partial_{x_i,t}^2\sigma(t,s,x)-\partial_{y_i,t}^2\sigma(t,s,y)\right|\leq K_N\left|x-y\right|^\delta,$\\$\forall |x|,|y|\leq N,\;\forall s,t\in[0,T],\; i=1\ldots d$,
\item $\left|\sigma(t_1,s,x)-\sigma(t_2,s,x)\right|+\left|\partial_{x_i}\sigma(t_1,s,x)-\partial_{x_i}\sigma(t_2,s,x)\right|
\leq K\left|t_1-t_2\right|^{\mu},$\\ $\forall
x\in\R^d,\;\forall t_1,t_2,s\in[0,T],\; i=1\ldots d$,
\item $\left|\sigma(t,s_1,x)-\sigma(t,s_2,x)\right|+\left|\partial_{t}\sigma(t,s_1,x)-\partial_{t}\sigma(t,s_2,x)\right|\leq K\left|s_1-s_2\right|^\beta,\\ \forall x\in\R^d,\;\forall s_1,s_2,t\in[0,T]$,
\item $\left|\partial_{x_i,t}^2\sigma(t,s_1,x)-\partial_{x_i,t}^2\sigma(t,s_2,x)\right|+\left|\partial_{x_i}\sigma(t,s_1,x)-\partial_{x_i}\sigma(t,s_2,x)\right| \leq K\left|s_1-s_2\right|^\beta$,\\$\forall x\in\R^d,\;\forall s_1,s_2,t\in[0,T],\, i=1,\ldots, d$.
\end{enumerate}
\item [\bfseries(H2)] $b:[0,T]^2\times \R^d \rightarrow \R^d$ is a measurable function such that there exists $b_0\in L^\rho([0,T]^2;\R^d)$ with $\rho\geq 2$, $0<\mu\leq 1$ and $\forall N\geq 0$ there exists $L_N>0$ such that:
\begin{enumerate}
\item $\left|b(t,s,x)-b(t,s,y)\right|\leq L_N\left|x-y\right|,\; \forall |x|,|y|\leq N,\,\forall s,t\in[0,T]$,
\item $\left|b(t_1,s,x)-b(t_2,s,x)\right|\leq L\left|t_1-t_2\right|^\mu,\; \forall x\in\R^d,\,\forall s,t_1,t_2\in[0,T]$,
\item $\left|b(t,s,x)\right|\leq L_0|x|+b_0(t,s),\quad \forall x\in\R^d,\;\forall s,t\in[0,T]$,
\item $\left|b(t_1,s,x_1)-b(t_1,s,x_2)-b(t_2,s,x_1)+b(t_2,s,x_2)\right|\leq L_N |t_1-t_2||x_1-x_2|,$\\
$\forall |x_1|,|x_2|\leq N,\;\forall t_1,t_2,s\in[0,T]$.
\end{enumerate}
\end{itemize}
\vskip 6pt
\begin{remark}
Actually, we can consider $\sigma$ and $b$ defined only in the set
$D \times \R^d$ with $D=\{(t,s)\in [0,T]^2; s \le t \}$.
\end{remark}

The following existence and uniqueness result holds.
\begin{theorem} \textnormal{\cite[Theorem 4.1]{BR}} \label{t:exist}
Assume that $\sigma$ and $b$ satisfy hypotheses ${\bf (H1)}$ and ${\bf (H2)}$ with $\rho=1/\alpha$, $\min\{\beta, \frac{\delta}{\1+\delta}\}>1-\mu$ and
$$
0<1-\mu<\alpha<\alpha_0:=\min\left\{\frac12, \beta, \frac{\delta}{1+\delta}\right\}.
$$
Then, equation \textnormal{(\ref{eq:det})} has a unique solution 
$x \in  C^{1-\alpha}(0,T;\R^d)$.
\end{theorem}

The first aim of this section is to obtain estimates for the supremum norm of the solution to \eqref{eq:det}. We first consider the case where $\sigma$ is bounded and the bound on $b$ does not depend on $x$.
\begin{theorem} \label{th:1}
Assume the hypotheses of Theorem \ref{t:exist} with $\mu=1$ and {\bf (H2)}\textnormal{(3)} replaced by 
\begin{equation} \label{e3p}
\left|b(t,s,x)\right|\leq L_0+b_0(t,s),\quad \forall x\in\R^d,\;\forall s,t\in[0,T].
\end{equation}
Assume that $\sigma$ is bounded.
Then, there exists a constant $C_{\alpha,\beta}>0$ such that
\begin{equation} \label{eq:result1}
\|x\|_{\infty}\leq |x_0|+1+T\left(\left(K_{T,\alpha}^{(1)}+K_{T,\alpha,\beta}^{(2)}\|g\|_{1-\alpha}\right)^{1/(1-\alpha)}\vee 1\vee T\right),
\end{equation}
where $K_{T,\alpha}^{(1)}=4(L(T\vee 1)+L_0+B_{0,\alpha})$ and $K_{T,\alpha,\beta}^{(2)}=C_{\alpha,\beta}(T+1+\|\sigma\|_\infty)$, 
$L,\,L_0$ are the constants in Hypothesis {\upshape\bfseries (H2)}, and $B_{0,\alpha}:=\sup_{t\in[0,T]}\lp\int_0^t |b_0(t,u)|^{1/\alpha}du\rp^\alpha$.
\end{theorem}

\begin{remark}
The techniques used in the proof do not seem to extend to the case $0<\mu<1$, thus it is left open for future work. More specifically, if $\mu<1$, the first term in equation (\ref{eq:cotaC1}) is of order
$i\Delta^{\mu+1-\alpha}$. Then, when dividing by $(t-s)^{1-\alpha}$ we obtain a term of order $i\Delta^{\mu}$ which cannot be bounded by $T$.
\end{remark}

\begin{proof}
We divide the interval $[0,T]$ into $n=[T/\Delta]+1$ subintervals, where $[a]$ denotes the largest integer strictly bounded by $a$ and $\Delta \leq 1$ will be chosen below.
\vskip 4pt
\noindent
{\it Step 1.} We start studying $\|x\|_{0,\Delta,1-\alpha}$. For $s,\,t\in[0,\Delta]$, $s<t$,
\begin{eqnarray}
|x_t-x_s|&\leq& \left|\int_0^s (b(t,r,x_r)-b(s,r,x_r))dr\right|+\left|\int_s^t b(t,r,x_r)dr\right|\label{eq:dec1}\\
&&+ \left| \int_0^s (\sigma(t,r,x_r)-\sigma(s,r,x_r))dg_r\right|+ \left| \int_s^t \sigma(t,r,x_r)dg_r\right| = A+B+C+D.\nonumber
\end{eqnarray}
\vskip 4pt
\noindent
Using the Hypothesis {\bf (H2)}(2), the term $A$ is easy to bound
\begin{equation} \label{eq:cotaA}
A\leq Ls(t-s).
\end{equation}
For the second term we use (\ref{e3p}) to obtain
\begin{equation} \label{eq:cotaB}
B\leq \left|\int_s^t\left( L_0+b_0(s,r)\right)dr \right| \leq L_0(t-s)+B_{0,\alpha}(t-s)^{1-\alpha}.
\end{equation}
\vskip 5pt
\noindent
For the next term, we use  \cite[Lemma A.2]{BR} to get
\begin{eqnarray*}
\left|D_{0+}^\alpha \left[ \sigma(t,\cdot,x_\cdot)-\sigma(s,\cdot,x_\cdot)\right] (r)\right|&\leq& \frac{K (t-s)}{\Gamma(1-\alpha)}
\lp \frac{1}{r^\alpha}\right.\left.+\alpha  \int_0^r \frac{(r-u)^\beta+|x_r-x_u|}{(r-u)^{\alpha+1}}du\rp\\
&\leq& C_{\alpha, \beta}(t-s)\left(r^{-\alpha}+ r^{\beta-\alpha}+\|x\|_{0,s,1-\alpha}r^{1-2\alpha}\right).
\end{eqnarray*}
Putting together the previous estimate, equation (\ref{eq:forpart}) and the estimate in \cite[(3.5)]{HN} we conclude that
\begin{eqnarray} \label{eq:cotaC}
C&\leq& C_{\alpha,\beta}\|g\|_{1-\alpha}(t-s) \left| \int_0^s \lp  r^{-\alpha}+r^{\beta-\alpha}+\|x\|_{0,s,1-\alpha}r^{1-2\alpha} \rp dr \right|\nonumber\\
&\leq& C_{\alpha,\beta} \|g\|_{1-\alpha}(t-s) \lp s^{1-\alpha}+s^{1+\beta-\alpha}+s^{2-2\alpha}\|x\|_{0,s,1-\alpha} \rp.
\end{eqnarray}
For term D, we obtain, proceeding similarly as for term $C$,
\begin{eqnarray*}
\left|D_{s+}^\alpha \left[ \sigma(t,\cdot,x_\cdot)\right] (r)\right|&\leq& \frac{1}{\Gamma(1-\alpha)}\lp {\|\sigma\|_\infty}(r-s)^{-\alpha}\right.\left.+\alpha K \int_s^r \frac{(r-u)^\beta+|x_r-x_u|}{(r-u)^{\alpha+1}}du\rp\\
&\leq& C_{\alpha,\beta}\lp \|\sigma\|_\infty(r-s)^{-\alpha}+ (r-s)^{\beta-\alpha}+\|x\|_{s,t,1-\alpha}(r-s)^{1-2\alpha}\rp.
\end{eqnarray*}
Therefore,
\begin{eqnarray}\label{eq:cotaD}
D&\leq& C_{\alpha,\beta}\|g\|_{1-\alpha} \int_s^t\lp\|\sigma\|_\infty(r-s)^{-\alpha}+ (r-s)^{\beta-\alpha}+\|x\|_{s,t,1-\alpha}(r-s)^{1-2\alpha}\rp dr\nonumber\\
&\leq& C_{\alpha,\beta}\|g\|_{1-\alpha}(t-s)^{1-\alpha} \lp \|\sigma\|_\infty+(t-s)^{\beta}+\|x\|_{s,t,1-\alpha}(t-s)^{1-\alpha}\rp.
\end{eqnarray}
Next, introducing \eqref{eq:cotaA}, \eqref{eq:cotaB}, \eqref{eq:cotaC} and \eqref{eq:cotaD} into (\ref{eq:dec1}), we obtain 
\begin{eqnarray*}
\frac{|x_t-x_s|}{(t-s)^{1-\alpha}}&\leq& L s(t-s)^{\alpha}+L_0(t-s)^{\alpha}+B_{0,\alpha}\\
&&+C_{\alpha,\beta}\|g\|_{1-\alpha}\lp (t-s)^{\alpha}\lp s^{1-\alpha}+s^{1+\beta-\alpha}+s^{2-2\alpha}\|x\|_{0,s,1-\alpha}\rp\right.\\
&&\left.+\|\sigma\|_\infty+(t-s)^{\beta}+\|x\|_{s,t,1-\alpha}(t-s)^{1-\alpha}\rp.
\end{eqnarray*}
Thus, 
\begin{eqnarray*} 
\|x\|_{0,\Delta,1-\alpha}&\leq& L\Delta^{1+\alpha}+L_0 \Delta^{\alpha}+B_{0,\alpha}\nonumber\\
&&+C_{\alpha,\beta}\|g\|_{1-\alpha}\lp \Delta+\Delta^{1+\beta}+\Delta^{\beta}+\|\sigma\|_\infty+\|x\|_{0,\Delta,1-\alpha}(\Delta^{2-\alpha}+\Delta^{1-\alpha})\rp \nonumber\\
&\leq& L+L_0+B_{0,\alpha}+C_{\alpha,\beta}\|g\|_{1-\alpha}\lp 1+\|\sigma\|_\infty+\|x\|_{0,\Delta,1-\alpha}\Delta^{1-\alpha}\rp,
\end{eqnarray*}
as $\Delta \leq 1$. Choosing $\Delta$ such that
\begin{equation} \label{d1}
\Delta^{1-\alpha} \leq \frac{1}{2C_{\alpha,\beta}\|g\|_{1-\alpha}},
\end{equation}
we obtain that 
\begin{equation}\label{eq:cotabeta1}
\|x\|_{0,\Delta,1-\alpha}\leq 2\lp L+L_0+B_{0,\alpha}+C_{\alpha,\beta}\|g\|_{1-\alpha}(1+\|\sigma\|_\infty)\rp.
\end{equation}
Therefore,
\begin{equation} \label{a1}
\|x\|_{0,\Delta,\infty }\leq |x_0|+\|x\|_{0,\Delta,1-\alpha}\Delta^{1-\alpha}\le |x_0|+\frac12,
\end{equation}
if $\Delta$ is such that
\begin{equation}\label{a2}
\Delta^{1-\alpha}\leq \frac{1}{4\lp L+L_0+B_{0,\alpha}+C_{\alpha,\beta}\|g\|_{1-\alpha}(1+\|\sigma\|_\infty)\rp}.
\end{equation}
\vskip 4pt
\noindent
{\it Step 2.} We next study $\|x\|_{s,t,1-\alpha}$ for $s,t\in [i\Delta, (i+1)\Delta]$, $s<t$. We write
\begin{eqnarray}
|x_t-x_s|&\leq& \left| \int_0^s\lp b(t,r,x_r)-b(s,r,x_r)\rp dr\right|+ \left| \int_s^t b(t,r,x_r) dr\right|\nonumber\\
&&+ \left| \int_0^{i\Delta}\lp \sigma(t,r,x_r)-\sigma(s,r,x_r)\rp dg_r\right|+ \left| \int_{i\Delta}^s\lp \sigma(t,r,x_r)-\sigma(s,r,x_r)\rp dg_r\right|\nonumber\\&&+\left|\int_s^t \sigma(t,r,x_r)dg_r\right|=A+B+C_1^i+C_2^i+D. \label{eq:dec2}
\end{eqnarray}
The terms $A$, $B$, and $D$ can be bounded exactly as in Step 1. Thus, it suffices to bound the terms $C_1^i$ and $C_2^i$. We start with $C_1^i$. We write
\begin{equation*}
C_1^i \leq \sum_{\ell=1}^i\left\vert\int_{(\ell-1)\Delta}^{\ell\Delta}\lp\sigma(t,r,x_r)-\sigma(s,r,x_r)\rp dg_r\right\vert.
\end{equation*}
Using \cite[Lemma A.2]{BR}, we get
\begin{align*}
&\left|D_{(\ell-1)\Delta+}^\alpha\right. \left.\lc \sigma(t,\cdot,x_\cdot)-\sigma(s,\cdot,x_\cdot)\rc(r)\right|\\
& \quad\leq \frac{K(t-s)}{\Gamma(1-\alpha)}\lp\frac{1}{(r-(\ell-1)\Delta)^\alpha}+\alpha\int_{(\ell-1)\Delta}^r \frac{\lp (r-u)^\beta+|x_r-x_u|\rp}{(r-u)^{\alpha+1}}du\rp\\
&\quad \leq C_{\alpha,\beta}\frac{(t-s)}{(r-(\ell-1)\Delta)^{\alpha}} \lp 1+(r-(\ell-1)\Delta)^{\beta}+(r-(\ell-1)\Delta)^{1-\alpha}\|x\|_{(\ell-1)\Delta,\ell\Delta,1-\alpha}\rp.
\end{align*}
Then, by the estimate in \cite[(3.5)]{HN}, we obtain 
\begin{equation} \label{eq:cotaC1}
C_1^i\leq C_{\alpha,\beta} \|g\|_{1-\alpha} (t-s) \sum_{\ell=1}^i
\lp \Delta^{1-\alpha}
+\Delta^{1+\beta-\alpha}+\Delta^{2-2\alpha}\|x\|_{(\ell-1)\Delta,\ell\Delta,1-\alpha}\rp.
\end{equation}
Similarly, for the term $C_2^i$ we obtain
\begin{equation} \label{eq:cotaC2}
\begin{split}
C_2^i&\leq C_{\alpha,\beta}\|g\|_{1-\alpha} (t-s) \int_{i\Delta}^s \frac{1}{(r-i\Delta)^\alpha}\lp 1+(r-i\Delta)^\beta+(r-i\Delta)^{1-\alpha}\|x\|_{i\Delta,s,1-\alpha} \rp dr \\
&\leq C_{\alpha,\beta}\|g\|_{1-\alpha} (t-s)\lp (s-i\Delta)^{1-\alpha}+(s-i\Delta)^{1+\beta-\alpha}+(s-i\Delta)^{2-2\alpha}\|x\|_{i\Delta,s,1-\alpha}\rp.
\end{split}
\end{equation}
Hence, from \eqref{eq:cotaA}, \eqref{eq:cotaB}, \eqref{eq:cotaD}, \eqref{eq:cotaC1} and \eqref{eq:cotaC2}, and using the fact that $\Delta\leq 1$, $t-s\leq \Delta$ and $i\Delta\leq T$, we obtain
\begin{align*}
\frac{|x_t-x_s|}{(t-s)^{1-\alpha}}&\leq Ls(t-s)^{\alpha}+L_0(t-s)^{\alpha}+B_{0,\alpha}\nonumber\\
&\qquad +C_{\alpha,\beta}\|g\|_{1-\alpha}\Bigg[ (t-s)^{\alpha} \sum_{\ell=1}^i \lp\Delta^{1-\alpha} +\Delta^{1+\beta-\alpha}+\Delta^{2-2\alpha}\|x\|_{(\ell-1)\Delta,\ell\Delta,1-\alpha}\rp  \nonumber\\
&\qquad +(t-s)^{\alpha}\lp (s-i\Delta)^{1-\alpha}+(s-i\Delta)^{1+\beta-\alpha}+(s-i\Delta)^{2-2\alpha}\|x\|_{i\Delta,s,1-\alpha}\rp \nonumber \\
&\qquad +(t-s)^{\beta}+\|\sigma\|_\infty + \|x\|_{s,t,1-\alpha}(t-s)^{1-\alpha}\Bigg]\nonumber\\
&\leq  LT+L_0+B_{0,\alpha}+C_{\alpha,\beta}\|g\|_{1-\alpha}\bigg[ T+1+\|\sigma\|_\infty\\
&\qquad +\Delta^{2-\alpha}\sum_{\ell=1}^{i} \|x\|_{(\ell-1)\Delta,\ell\Delta,1-\alpha}+\Delta^{1-\alpha}\|x\|_{i\Delta,(i+1)\Delta,1-\alpha}\bigg].\nonumber
\end{align*}
Choosing $\Delta$ such that
\begin{equation} \label{d6}
\Delta^{1-\alpha} \leq \frac{1}{2C_{\alpha,\beta}\|g\|_{1-\alpha}},
\end{equation}
we obtain that 
\begin{eqnarray} \label{xnorm}
\|x\|_{i\Delta,(i+1)\Delta,1-\alpha}\leq A_1+A_2\Delta^{2-\alpha}\sum_{\ell=1}^{i} \|x\|_{(\ell-1)\Delta,\ell\Delta,1-\alpha},
\end{eqnarray}
where
\begin{eqnarray*}
A_1&=&  2(  LT+L_0+B_{0,\alpha}+
C_{\alpha,\beta}\|g\|_{1-\alpha}(T+1+\|\sigma\|_\infty)),\\
A_2&=& 2C_{\alpha,\beta}\|g\|_{1-\alpha}.
\end{eqnarray*} 

\noindent {\it Step 3.} We now use an induction argument in order to show that for all $i\geq 0$, $$\Delta^{1-\alpha} \|x\|_{i\Delta,(i+1)\Delta,1-\alpha}\leq 1.$$
For $i=0$ it is proved in Step 1. 
Assuming that it is true up to $i-1$ and using (\ref{xnorm}), we get that
\begin{equation*}
\Delta^{1-\alpha} \|x\|_{i\Delta,(i+1)\Delta,1-\alpha}\leq A_1\Delta^{1-\alpha}+A_2\Delta^{3-2\alpha}\sum_{\ell=1}^i\|x\|_{(\ell-1)\Delta,\ell \Delta,1-\alpha}\leq  \Delta^{1-\alpha}(A_1+A_2T).
\end{equation*}
Finally, it suffices to choose $\Delta$ such that
\begin{equation} \label{b1}
\Delta^{1-\alpha}\leq \frac{1}{A_1+A_2T},
\end{equation}
to conclude the desired claim.

Therefore, we have that
\begin{equation}  \label{b2}
\|x\|_{i\Delta,(i+1)\Delta,\infty}\leq|x_{i\Delta}|+\Delta^{1-\alpha}  \|x\|_{i\Delta,(i+1)\Delta,1-\alpha}\leq |x_{i\Delta}|+1. 
\end{equation}
Applying this inequality recursively, we conclude that 
\[\sup_{0 \leq t \leq T}\vert x_t \vert \leq\sup_{0 \leq t \leq (n-1) \Delta} \vert x_t \vert +1\leq \cdots\leq\vert  x_0\vert +n,\]
and the desired bound follows choosing $\Delta$ such that
$$
\Delta
= \frac{1}{(4(  L(T\vee 1)+L_0+B_{0,\alpha}+
C_{\alpha,\beta}\|g\|_{1-\alpha}(T+1+\|\sigma\|_\infty)))^{1/(1-\alpha)}} \wedge 1 \wedge T,
$$
where $C_{\alpha, \beta}$ is such that (\ref{d1}), (\ref{a2}), (\ref{d6}) and (\ref{b1}) hold.
\end{proof}

The next result is an exponential bound for the supremum norm of the solution to (\ref{eq:det}) under more general hypotheses than the previous theorem.
\begin{theorem} \label{th:2}
Assume the hypotheses of Theorem \ref{t:exist} with $\mu=1$. Then, 
there exists a constant $C_{\alpha,\beta}>0$ such that
\begin{equation*} \label{eq:resultth2}
\|x\|_{\infty}\leq \lp|x_0|+1\rp \exp\left( 2T \left( \left( K_{T,\alpha}^{(3)}+K_{T,\alpha,\beta}^{(4)}\|g\|_{1-\alpha}\right)^{1/(1-\alpha)}\vee 1\vee T\right)\right),
\end{equation*}
where $K_{T,\alpha}^{(3)}=6(L_0+L(T+1)+B_{0,\alpha})$, $K_{T,\alpha,\beta}^{(4)}=C_{\alpha,\beta}(T+1)$, and
$L,\,L_0$, and $B_{0,\alpha}$ are as in Theorem \ref{th:1}.
\end{theorem}

\begin{proof}
The proof follows similarly as the proof of Theorem \ref{th:1}.
We divide the interval $[0,T]$ into $n=[T/\Delta]+1$ subintervals, where $\Delta \leq 1$ will be chosen below.
\vskip 4pt
\noindent
{\it Step 1.} We start bounding $\|x\|_{0,\Delta,1-\alpha}$.
We can use the same bound for
$|x_t-x_s|$ obtained in \eqref{eq:dec1}. Then, terms $A$  and $C$ can be bounded as in \eqref{eq:cotaA} and \eqref{eq:cotaC} respectively. 
For term $B$, using {\bf (H2)}(3), we get that
\begin{equation} \label{b}
B\leq L_0(t-s)\|x\|_{s,t,\infty}+B_{0,\alpha}(t-s)^{1-\alpha}.
\end{equation}
For term $D$, we obtain
\begin{equation} \label{d}
D\leq C_{\alpha,\beta}\|g\|_{1-\alpha} (t-s)^{1-\alpha}\lc \|x\|_{s,t,\infty}+(t-s)^{\beta}+(t-s)^{1-\alpha}\|x\|_{s,t,1-\alpha}\rc.
\end{equation}
Thus, we get that
\begin{eqnarray*} 
\frac{|x_t-x_s|}{(t-s)^{1-\alpha}}&\leq& Ls(t-s)^{\alpha}+B_{0,\alpha}+\|x\|_{s,t,\infty}\lc L_0(t-s)^{\alpha}+C_{\alpha,\beta}\|g\|_{1-\alpha}\rc\nonumber\\&&+C_{\alpha,\beta}\|g\|_{1-\alpha} (t-s)^{\alpha} \lc s^{1-\alpha}+s^{1+\beta-\alpha}+s^{2-2\alpha} \|x\|_{0,s,1-\alpha}\rc\nonumber\\
&&+C_{\alpha,\beta}\|g\|_{1-\alpha}\lc (t-s)^{\beta}+(t-s)^{1-\alpha}\|x\|_{s,t,1-\alpha}\rc.\nonumber
\end{eqnarray*}
Hence,  as $\Delta \leq 1$,
\begin{eqnarray*} 
\|x\|_{0,\Delta,1-\alpha}\leq B_0 +B_1\|x\|_{0,\Delta,\infty}+B_2\|x\|_{0,\Delta,1-\alpha},
\end{eqnarray*}
where
\begin{eqnarray*}
B_0&=&L+B_{0,\alpha}+ C_{\alpha,\beta}\|g\|_{1-\alpha},\\
B_1&=& L_0+C_{\alpha,\beta}\|g\|_{1-\alpha},\\
B_2&=&  \Delta^{1-\alpha} C_{\alpha,\beta}\|g\|_{1-\alpha}.
\end{eqnarray*}
Thus,
\begin{equation} \label{gp}
\|x\|_{0,\Delta,1-\alpha}\leq B_0 (1-B_2)^{-1}+B_1 (1-B_2)^{-1}\|x\|_{0,\Delta,\infty}.
\end{equation}
Therefore, using the fact that
$$
\sup_{t\in[0,\Delta]}|x_t|\leq |x_0|+\|x\|_{0,\Delta,1-\alpha}\Delta^{1-\alpha},
$$
we conclude that
\begin{equation} \label{eq:cotainfinit3}
\sup_{t\in[0,\Delta]}|x_t|
\leq B_3^{-1}|x_0|+B_3^{-1}B_0 (1-B_2)^{-1}\Delta^{1-\alpha},
\end{equation}
where $B_3=1-B_1 (1-B_2)^{-1}\Delta^{1-\alpha}$.
\vskip 4pt
\noindent
{\it Step 2.} We next study $\|x\|_{i\Delta,(i+1)\Delta,\beta}$, for $i \geq 0$.
For $s,\,t\in[i\Delta,(i+1)\Delta]$, $s<t$,
$|x_t-x_s|$ can be bounded as in \eqref{eq:dec2}. Then using  \eqref{eq:cotaA}, \eqref{eq:cotaC1}, \eqref{eq:cotaC2}, \eqref{b}, and \eqref{d}, we get that
\begin{align*}
\frac{|x_t-x_s|}{(t-s)^{1-\alpha}}&\leq Ls(t-s)^{\alpha}+L_0(t-s)^{\alpha}\|x\|_{s,t,\infty}+B_{0,\alpha}\nonumber\\
&\qquad +C_{\alpha,\beta}\|g\|_{1-\alpha}\Bigg[ (t-s)^{\alpha} \sum_{\ell=1}^i \lp\Delta^{1-\alpha} +\Delta^{1+\beta-\alpha}+\Delta^{2-2\alpha}\|x\|_{(\ell-1)\Delta,\ell\Delta,1-\alpha}\rp  \nonumber\\
&\qquad +(t-s)^{\alpha}\lp (s-i\Delta)^{1-\alpha}+(s-i\Delta)^{1+\beta-\alpha}+(s-i\Delta)^{2-2\alpha}\|x\|_{i\Delta,s,1-\alpha}\rp \nonumber \\
&\qquad +(t-s)^{\beta}+\|x\|_{s,t,\infty} + \|x\|_{s,t,1-\alpha}(t-s)^{1-\alpha}\Bigg]\\
&\leq  LT+L_0\|x\|_{i\Delta,(i+1)\Delta,\infty}+B_{0,\alpha}+C_{\alpha,\beta}\|g\|_{1-\alpha}\bigg[ T+1+\|x\|_{i\Delta,(i+1)\Delta,\infty}\\
&\qquad+\Delta^{2-\alpha}\sum_{\ell=1}^{i} \|x\|_{(\ell-1)\Delta,\ell\Delta,1-\alpha}+\Delta^{1-\alpha}\|x\|_{i\Delta,(i+1)\Delta,1-\alpha}\bigg].\nonumber
\end{align*}
Therefore, we obtain that
\begin{equation} \label{gp2}
\|x\|_{i\Delta,(i+1)\Delta,1-\alpha}\leq C_0^{-1} \lc C_1+C_2\|x\|_{i\Delta,(i+1)\Delta,\infty}+C_3 \Delta^{2-\alpha}\sum_{\ell=1}^i\|x\|_{(\ell-1)\Delta,\ell\Delta,1-\alpha}\rc,
\end{equation}
where
\begin{eqnarray*}
C_0&=&1-C_{\alpha,\beta}\|g\|_{1-\alpha}\Delta^{1-\alpha},\\
C_1&=& LT+B_{0,\alpha}+C_{\alpha,\beta}\|g\|_{1-\alpha} \lc T+1\rc,\\
C_2&=& L_0+C_{\alpha,\beta}\|g\|_{1-\alpha},\\
C_3&=& C_{\alpha,\beta}\|g\|_{1-\alpha}.
\end{eqnarray*} 
Thus,
\begin{equation} \label{gp3}
\|x\|_{i\Delta,(i+1)\Delta,\infty}\leq C_4^{-1}|x_{i\Delta}|+C_0^{-1}C_4^{-1}\Delta^{1-\alpha}\big( C_1+C_3\Delta^{2-\alpha}\sum_{\ell=1}^i\|x\|_{(\ell-1)\Delta,\ell\Delta,1-\alpha}\big),
\end{equation}
where $C_4=1-C_0^{-1}C_2\Delta^{1-\alpha}$.

We next show by induction that for all $i\geq 0$, $$\Delta^{1-\alpha} \|x\|_{i\Delta,(i+1)\Delta,1-\alpha}\leq 1+\|x\|_{0,(i+1)\Delta,\infty}.$$
For $i=0$ it is proved in (\ref{gp}) that
\begin{equation*} 
\|x\|_{0,\Delta,1-\alpha}\leq B_0 (1-B_2)^{-1}+B_1 (1-B_2)^{-1}\|x\|_{0,\Delta,\infty}.
\end{equation*}
Then, it suffices to choose $\Delta$ such that $B_2 \leq \frac12$ and
$$
\Delta^{1-\alpha}\leq \frac12 \left(\frac{1}{B_0} \wedge \frac{1}{B_1}\right),
$$
to conclude the claim for $i=0$.

Assuming that it is true up to $i-1$ and using (\ref{gp2}), we get that
\begin{equation*}
 \|x\|_{i\Delta,(i+1)\Delta,1-\alpha}\leq C_0^{-1} \left[ C_1+C_3T+\|x\|_{0,(i+1)\Delta,\infty}(C_2+C_3T)\right].
\end{equation*}
Finally, it suffices to choose $\Delta$ such that $C_0\geq 2$ and
$$
\Delta^{1-\alpha}\leq \frac{1}{C_1+C_3T} \wedge \frac{1}{C_2+C_3T},
$$
to conclude the desired claim.

By (\ref{gp3}), we conclude that
\begin{equation} \label{eq:normst2}
\|x\|_{i\Delta,(i+1)\Delta,\infty}\leq C_4^{-1}|x_{i\Delta}|+C_0^{-1}C_4^{-1}\Delta^{1-\alpha}\lp C_1+T C_3(1+\|x\|_{0,i\Delta,\infty})\rp.
\end{equation}
\vskip 4pt
\noindent
{\it Step 3.} 
Using \eqref{eq:normst2}, we get that
\begin{align*}
&\sup_{0\leq t\leq (i+1)\Delta} |x_t|\leq  \sup_{0\leq t\leq i\Delta} |x_t| + \sup_{i \Delta \leq t\leq (i+1)\Delta} |x_t|\\
&\quad\le  \sup_{0\leq t\leq i\Delta} |x_t| +C_4^{-1}|x_{i\Delta}|+C_0^{-1}C_4^{-1}\Delta^{1-\alpha} \lp C_1+T C_3(1+\|x\|_{0,i\Delta,\infty})\rp\\
&\quad\le  K_1 \sup_{0\leq t\leq i\Delta}|x_t|+K_2,
\end{align*}
where
\begin{eqnarray*}
K_1&=& 1+C_4^{-1}(1+ T C_0^{-1} C_3 \Delta^{1-\alpha}),\\
K_2&=&C_0^{-1}C_4^{-1}(C_1+TC_3)\Delta^{1-\alpha}.
\end{eqnarray*}
Iterating, we obtain that
\[ \sup_{0 \leq t \leq T}\vert x_t \vert \leq K_1\sup_{0 \leq t \leq (n-1) \Delta} \vert x_t \vert +K_2\leq \cdots\leq K_1^{n-1} \sup_{0 \leq t \leq \Delta}\vert x_t \vert+K_2\sum_{i=0}^{n-2}K_1^i.\]

We next choose $\Delta$ such that $C_2\Delta^{1-\alpha}\leq \frac13$ and $C_0^{-1}\leq \frac32$.
Then, $C_4^{-1}\leq 2$. Moreover, we choose $\Delta$  such that 
$T C_3C_0^{-1} \Delta^{1-\alpha}\leq \frac16$. This  implies that $K_1\leq \frac{10}{3}$. Thus, 
$$\sum_{i=0}^{n-2}K_1^i \leq \sum_{i=0}^{n-2}\lp\frac{10}{3}\rp^i=\frac37\lp\frac{10}{3}\rp^{n-1}\leq \;\frac37 e^{2(n-1)}.$$
In order to bound $K_2$, it suffices to choose $\Delta$ such  that $C_1\Delta^{1-\alpha}\leq \frac13$. Then, we easily obtain that $K_2\leq 1$.
We finally bound $\sup_{0 \leq t \leq \Delta}\vert x_t \vert$ using  \eqref{eq:cotainfinit3}. Again  we choose $\Delta$ such that $(1-B_2)^{-1}\leq\frac32$ and $B_1\Delta^{1-\alpha}\leq \frac{1}{3}$ so that $B_3^{-1}\leq 2$. We also choose $\Delta$ such that $\Delta^{1-\alpha} B_0\leq \frac{1}{4}$ so
that \[\sup_{0 \leq t \leq \Delta} \vert x_t \vert\leq 2|x_0|+2\cdot \frac{3}{2}\cdot\frac14=2|x_0|+\frac{3}{4}<2|x_0|+1.\]
Finally, we conclude that
$$
\sup_{t\in[0,T]} |x_t|\leq (2|x_0|+1)e^{2(n-1)}\leq (|x_0|+1)e^{2\left[\frac{T}{\Delta}\right]},
$$
which implies the desired estimate choosing $\Delta$ such that 
\[\Delta= \frac{1}{(C_{\alpha,\beta}\|g\|_{1-\alpha}(1+T)+6(L_0+L(1+T)+B_{0,\alpha}))^{1/(1-\alpha)}} \wedge 1 \wedge T.\]
\end{proof}

The next result provides a supremum norm estimate of the solution $z_{t}$ of the following system of equations 
\begin{align} \label{z} \nonumber 
x_t=& x_0+\displaystyle\int_0^t b(t,r,x_r)dr+\displaystyle\int_0^t \sigma(t,r,x_r)dg_r\\ 
z_t=& w_t+\displaystyle\int_0^t h(t,r,x_r)z_rdr+\displaystyle\int_0^t f(t,r,x_r)z_rdg_r,
\end{align}
where $g$ belongs to $W_2^{1-\alpha}(0,T; \R^m)$, $w$ belongs to $C^{1-\alpha}(0,T;\R^d)$, $b:[0,T]^2\times \R^d \rightarrow \R^d$, $\sigma:[0,T]^2\times \R^d \rightarrow \R^d\times \R^m$, $h:[0,T]^2\times \R^d \rightarrow \R^d\times \R^d$, and $f:[0,T]^2\times \R^d \rightarrow \R^{d^2}\times \R^{m}$ are measurable functions,  and $x_0\in \R^d$. 
\vskip 5pt
\noindent
We will use the following hypotheses on $h,f$ and $w$:
\[
\arraycolsep=1.4pt\def\arraystretch{1.5}{\textbf{(H3)}}\quad\left\{
\begin{array}{l}
h \;\textrm{is Lipschitz continuous with respect to}\; t \textrm{ and bounded}.\\
f\; \textrm{is bounded and satisfies {\textbf{(H1)}}}. \\
w\; \textrm{is Lipschitz continuous and bounded}.
\end{array}\right.\]
\vskip 5pt
\begin{theorem} \label{th:3}
Assume that $b$ and $\sigma$ satisfy  the hypotheses of Theorem \ref{th:1} and that $h, f$ and $w$ satisfy hypothesis {\upshape\bfseries(H3)}. 
Then there exists a unique solution 
$z \in C^{1-\alpha}(0,T;\R^d)$ to equation \eqref{z}. Moreover, there exists a constant $C_{\alpha,\beta}>0$ such that 
\[\|z\|_\infty \leq 2\left(1+\|w\|_\infty\right)\exp\left(T\left(K_{T,\alpha}^{(5)}+K_{T,\alpha,\beta}^{(6)}\|g\|_{1-\alpha} \right)^{1/(1-\alpha)}\vee 1 \vee T \right),\]
where $K_{T,\alpha}^{(5)}=16\left(K+\|h\|_\infty+L+L_0+B_{0,\alpha}\right)e^T(T+1)$ and \\$K_{T,\alpha,\beta}^{(6)}=C_{\alpha,\beta}\left(\|f\|_\infty+\|\sigma\|_\infty+1\right)e^T (T+1)$.
\end{theorem} 
\begin{proof}
The  existence and uniqueness of the solution follows similarly as \cite[Theorem 5.1]{N-R}.
We next prove the estimate of the supremum norm of the  solution.
We divide the interval $[0,T]$ into $n=[T/\tilde{\Delta}]+1$ subintervals, where $\tilde{\Delta} \leq 1$ will be chosen below.
\vskip 4pt
\noindent
{\it{Step 1.}} We first estimate $\|z\|_{0,\tilde{\Delta},\infty}$. Let $t,t'\in [0,\tilde{\Delta}]$ with $t<t'$. We write
\begin{align*}
\left|z_{t'}-z_t\right|&\leq \left|w_{t'}-w_t\right|+\left|\int_t^{t'} h(t',r,x_r)z_rdr\right|+\left|\int_0^t \left(h(t',r,x_r)-h(t,r,x_r)\right)z_rdr\right|\\
&\qquad \qquad\qquad   +\left|\int_t^{t'} f(t',r,x_r)z_rdg_r\right|+\left|\int_0^t \left(f(t',r,x_r)-f(t,r,x_r)\right)z_rdg_r\right|\\
&=E+F+G+H+I.
\end{align*}
The first three terms are easily bounded as
\begin{align*}
E&\leq K(t'-t), \\
F &\leq \|h\|_\infty (t'-t) \|z\|_{t,t',\infty},\text{ and} \\
G &\leq L(t'-t) t\|z\|_{t,t',\infty}. 
\end{align*}
We next bound $H$ and $I$. Using \eqref{eq:forpart} and the estimate in \cite[(3.5)]{HN}, we  get 
\[H\leq K\|g\|_{1-\alpha} \int_t^{t'} \left|D_{t+}^\alpha\left[f(t',\cdot,x_\cdot)z_\cdot\right](r)\right|dr,\]
where
\begin{align*}
&\vert D_{t+}^\alpha\left[f(t',\cdot,x_\cdot)z_\cdot\right](r)\vert \\
&\qquad \leq \frac{1}{\Gamma(1-\alpha)}\left(\frac{|f(t',r,x_r)z_r|}{(r-t)^\alpha}+\alpha \int_t^r \frac{|f(t',r,x_r)z_r-f(t',u,x_u)z_u|}{(r-u)^{\alpha+1}}du\right)\\
&\qquad \leq  C_\alpha (H_1+H_2),
\end{align*}
\[H_1\leq C_{\alpha}\|f\|_\infty \|z\|_{t,t',\infty} (r-t)^{-\alpha},\]
and
\begin{eqnarray*}
H_2&\leq& C_{\alpha,\beta} \|z\|_{t,t',\infty} (r-t)^{\beta-\alpha}+ C_{\alpha,\beta}\|z\|_{t,t',\infty}\|x\|_{t,t',1-\alpha}(r-t)^{1-2\alpha}\\
&&+C_{\alpha,\beta}\|f\|_\infty\|z\|_{t,t',1-\alpha}(r-t)^{1-2\alpha}.
\end{eqnarray*}
Therefore, we obtain
\begin{eqnarray*}
H&\leq& C_{\alpha,\beta} (t'-t)^{1-\alpha} \|g\|_{1-\alpha} \left[ \|z\|_{t,t',\infty}
\left(\|f\|_\infty+(t'-t)^\beta+\|x\|_{t,t',1-\alpha}(t'-t)^{1-\alpha}\right)\right. \nonumber\\
&&\hspace{5mm}\left.+\|f\|_\infty\|z \|_{t,t',1-\alpha}\|(t'-t)^{1-\alpha}\right].
\end{eqnarray*}
Similarly,
\begin{equation*} 
I\leq C_{\alpha,\beta}\|g\|_{1-\alpha} (t'-t) t^{1-\alpha} \left[ \|z\|_{t,t',\infty}\left(1+t^\beta+
\|x\|_{t,t',1-\alpha} t^{1-\alpha}\right)+\|z\|_{t,t',1-\alpha}t^{1-\alpha}\right].
\end{equation*}
Hence, we conclude that
\begin{eqnarray*}
\frac{|z_{t'}-z_t|}{(t'-t)^{1-\alpha}}&\leq& K+D_1\|z\|_{t,t',\infty}+D_2\|z\|_{t,t',1-\alpha},
\end{eqnarray*}
where
\begin{eqnarray*}
D_1&=& \|h\|_\infty+L+C_{\alpha,\beta}\|g\|_{1-\alpha}\left(\|f\|_\infty+1+\|x\|_{0,\tilde{\Delta},1-\alpha}\tilde{\Delta}^{1-\alpha}\right),\\
D_2 &=& C_{\alpha,\beta} \|g\|_{1-\alpha}\left(\|f\|_\infty+1\right)\tilde{\Delta}^{1-\alpha}.
\end{eqnarray*}
Thus,
\[\|z\|_{0,\tilde{\Delta},1-\alpha}\leq (1-D_2)^{-1}(K+D_1\|z\|_{0,\tilde{\Delta},\infty}).\]
Moreover,
\begin{eqnarray*}
\|z\|_{0,\tilde{\Delta},\infty}
\leq  \|w\|_\infty+\tilde{\Delta}^{1-\alpha}\left[ (1-D_2)^{-1}(K+  D_1 \|z\|_{0,\tilde{\Delta},\infty})\right].
\end{eqnarray*}
Choosing $\tilde{\Delta}$ satisfying (\ref{a2}), we obtain by (\ref{a1}) that $\|x\|_{0,\tilde{\Delta},1-\alpha}\tilde{\Delta}^{1-\alpha}\leq \frac12 \leq 1$. We next choose $\tilde{\Delta}$ such that $\tilde{\Delta}^{1-\alpha} K\leq  1$, $\tilde{\Delta}^{1-\alpha} D_1 \leq 1$, and $D_2\leq \frac12$. Then, we obtain that
\begin{equation} \label{w1}
\|z\|_{0,\tilde{\Delta},\infty} \leq 2\|w\|_\infty+1.
\end{equation}
{\it{Step 2.}} We next estimate $\|z\|_{i\tilde{\Delta},(i+1)\tilde{\Delta},\infty}$ for $i=1,\ldots,n$. Fix $t,t'\in [i\tilde{\Delta},(i+1)\tilde{\Delta}]$ with $t<t'$.
Similar bounds can be obtained for the corresponding terms $E,\,F,\,G$ and $H$ as in Step 1. Thus, we just need to bound the term $I^i:= I$, that is,
\begin{equation*}
I^i \leq \sum_{\ell=1}^i\left|\int_{(\ell-1)\tilde{\Delta}}^{\ell\tilde{\Delta}} \left(f(t',r,x)-f(t,r,x_r)\right)z_rdg_r\right|+\left|\int_{i\tilde{\Delta}}^{t} \left(f(t',r,x)-f(t,r,x_r)\right)z_rdg_r\right|.
\end{equation*}
Following the same computations as for $I$, we get
\begin{eqnarray*}
\left|\int_{(\ell-1)\tilde{\Delta}}^{\ell\tilde{\Delta}} \left(f(t',r,x)-f(t,r,x_r)\right)z_rdg_r\right| &\leq& C_{\alpha,\beta}\|g\|_{1-\alpha} (t'-t) \tilde{\Delta} \left[\|z\|_{\ell-1)\tilde{\Delta},\ell\tilde{\Delta},1-\alpha}\tilde{\Delta}^{1-\alpha}\right].\\
&&\hspace{-20mm}\left.+\|z\|_{(\ell-1)\tilde{\Delta},\ell\tilde{\Delta},\infty} \left(1+\tilde{\Delta}^{1-\alpha} + \|x\|_{(\ell-1)\tilde{\Delta},\ell\tilde{\Delta},1-\alpha}\tilde{\Delta}^{1-\alpha}\right)\right].
\end{eqnarray*}
Therefore, the term $I^i$ is bounded by
\begin{align*}
&C_{\alpha,\beta} \|g\|_{1-\alpha} (t'-t) \tilde{\Delta}^{1-\alpha}
\bigg[ \|z\|_{i\tilde{\Delta},(i+1)\tilde{\Delta},\infty}\left(1+\|x\|_{i\tilde{\Delta},(i+1)\tilde{\Delta},1-\alpha}\tilde{\Delta}^{1-\alpha}\right)\\
&+\|z\|_{i\tilde{\Delta},(i+1)\tilde{\Delta},1-\alpha}\tilde{\Delta}^{1-\alpha}\\
&+\sum_{\ell=1}^i \Big[\|z\|_{(\ell-1)\tilde{\Delta},\ell\tilde{\Delta},\infty}\left(1+\|x\|_{(l-1)\tilde{\Delta},\ell\tilde{\Delta},1-\alpha}\tilde{\Delta}^{1-\alpha}\right)+\|z\|_{(\ell-1)\tilde{\Delta},\ell\tilde{\Delta},1-\alpha}\tilde{\Delta}^{1-\alpha}\Big]\bigg]. 
\end{align*}
Hence, we obtain that
\begin{align*}
\|z\|_{i\tilde{\Delta},(i+1)\tilde{\Delta},1-\alpha}&\leq K+E_1\|z\|_{i\tilde{\Delta},(i+1)\tilde{\Delta},1-\alpha} +E_2^i\|z\|_{i\tilde{\Delta},(i+1)\tilde{\Delta},\infty}\\
&\qquad +\sum_{\ell=1}^i \left[E_3^\ell\|z\|_{(\ell-1)\tilde{\Delta},\ell\tilde{\Delta},\infty})+E_4\|z\|_{(\ell-1)\tilde{\Delta},\ell\tilde{\Delta},1-\alpha}\right].
\end{align*}
where
\begin{eqnarray*}
E_1&=&C_{\alpha}\|g\|_{1-\alpha}\left(\|f\|_\infty+1\right)\tilde{\Delta}^{1-\alpha},\\
E_2^i &= &\|h\|_\infty+L+C_{\alpha,\beta}\|g\|_{1-\alpha}\left(\|f\|_\infty+1+\|x\|_{i\tilde{\Delta},(i+1)\tilde{\Delta},1-\alpha}\tilde{\Delta}^{1-\alpha}\right),\\
E_3^\ell &=&  C_{\alpha,\beta}\|g\|_{1-\alpha} \tilde{\Delta} \left(1+\|x\|_{(\ell-1)\tilde{\Delta},\ell\tilde{\Delta},1-\alpha}\tilde{\Delta}^{1-\alpha}\right),\\
E_4 &=&  C_{\alpha,\beta}\|g\|_{1-\alpha} \tilde{\Delta}^{2-\alpha}.
\end{eqnarray*}
Choosing $\tilde{\Delta}$ such that $E_1 \leq \frac12$, we obtain that
\begin{equation} \label{aux3} \begin{split}
\|z\|_{i\tilde{\Delta},(i+1)\tilde{\Delta},1-\alpha}&\leq 2K+2E_2^i\|z\|_{i\tilde{\Delta},(i+1)\tilde{\Delta},\infty}\\
&\qquad +\sum_{\ell=1}^i \left[2 E_3^\ell\|z\|_{(\ell-1)\tilde{\Delta},\ell\tilde{\Delta},\infty})+2E_4\|z\|_{(\ell-1)\tilde{\Delta},\ell\tilde{\Delta},1-\alpha}\right].
\end{split}
\end{equation}
Choosing $\tilde{\Delta}$ satisfying (\ref{b1}), we obtain by the Step 3 in Theorem \ref{th:1} that for all $\ell=1,\ldots,i$, $\|x\|_{\ell\tilde{\Delta},(\ell+1)\tilde{\Delta},1-\alpha}\tilde{\Delta}^{1-\alpha} \leq 1$. Thus,
\begin{eqnarray*}
E_2^i &\leq& E_2:= \|h\|_\infty+L+C_{\alpha,\beta}\|g\|_{1-\alpha}\left(\|f\|_\infty+1\right)\\
E_3^\ell &\leq& E_3:= C_{\alpha,\beta}\|g\|_{1-\alpha} \tilde{\Delta}.
\end{eqnarray*}
Applying expression (\ref{aux3}) recurrently we obtain  that
\begin{align*}
\|z\|_{i\tilde{\Delta},(i+1)\tilde{\Delta},1-\alpha}&\leq 2K(1+2E_4)^{i-1}+ 2E_2 \|z\|_{i\tilde{\Delta},(i+1)\tilde{\Delta},\infty}\\
& + (2E_3+4E_4E_2) \sum_{\ell=1}^{i-1} (1+2E_4)^{\ell-1} \|z\|_{(i-\ell)\tilde{\Delta},(i-(\ell-1))\tilde{\Delta},\infty}.
\end{align*}
This implies that
\begin{eqnarray*}
\|z\|_{i\tilde{\Delta},(i+1)\tilde{\Delta},\infty}&\leq& \left|z_{i\tilde{\Delta}}\right|+ \|z\|_{i\tilde{\Delta},(i+1)\tilde{\Delta},1-\alpha}\tilde{\Delta}^{1-\alpha}\\
&\leq& E_5^{-1}\left|z_{i\tilde{\Delta}}\right|+ K_i + E_5^{-1}(2E_3+4E_4E_2)(1+2E_4)^{i-1} i \|z\|_{0,i\tilde{\Delta},\infty}\tilde{\Delta}^{1-\alpha},
\end{eqnarray*}
where $E_5=1-2E_2\tilde{\Delta}^{1-\alpha}$ and $K_i=E_5^{-1}2E_1(1+2E_4)^{i-1}\tilde{\Delta}^{1-\alpha}$.
\vskip 4pt
\noindent
{\it Step 3.} Using the  result of Step 2 yields that
\begin{equation} \label{w2}
\sup_{t\in[0,(i+1)\tilde{\Delta}]} |z_t|\leq L_i \sup_{t\in[0,i\tilde{\Delta}]} |z_t| +K_i, 
\end{equation}
where $
L_i =  E_5^{-1}\left(1+\tilde{\Delta}^{1-\alpha} (2E_3+4E_4E_2)(1+2E_4)^{i-1} i\right)
$.

We finally bound $L_i$ and $K_i$.
We choose $\tilde{\Delta}$ such that $2E_2\tilde{\Delta}^{1-\alpha}\leq \frac12$, so that $E_5^{-1}\leq 2$. We also choose $\tilde{\Delta}$ such that $2E_4 \leq \tilde{\Delta}$ so that
\begin{equation*}
(1+2E_4)^{i-1}\leq (1+\tilde{\Delta})^{i-1}\leq (1+\tilde{\Delta})^{n-1}\leq (1+\tilde{\Delta})^{T/\tilde{\Delta}}\leq e^{T}.
\end{equation*}
Hence, choosing $\tilde{\Delta}$ such that $4 E_1 e^T \tilde{\Delta}^{1-\alpha} \leq 1$ we conclude that 
$K_i\leq 1.$
Moreover, as $i\tilde{\Delta}\leq T$, we have that
\[L_i\leq 2\left(1+ \left(\tilde{\Delta}^{1-\alpha} C_{\alpha,\beta} \Vert g \Vert_{1-\alpha}+2\tilde{\Delta}^{1-\alpha} E_2\right)e^{T}T\right).\]
We finally choose $\tilde{\Delta}$ such that that  $\tilde{\Delta}^{1-\alpha} C_{\alpha,\beta} \Vert g \Vert_{1-\alpha} e^T T\leq \frac18$ and 
$2\tilde{\Delta}^{1-\alpha} E_2 e^T T \leq \frac18$, so that
$L_i\leq e$.

Iterating (\ref{w2}) and using (\ref{w1}), we conclude that
\begin{equation*} \begin{split}
\sup_{0 \leq t \leq T}\vert z_t \vert &\leq e \sup_{0 \leq t \leq (n-1) \tilde{\Delta}} \vert z_t \vert +1\leq \cdots\leq e^{n-1} \sup_{0 \leq t \leq \tilde{\Delta}}\vert z_t \vert+\sum_{i=0}^{n-2}e^i\\
&\leq 2e^{[T/\tilde{\Delta}]}(\Vert w \Vert_{\infty}+1),
\end{split}
\end{equation*}
which implies  the desired result.
\end{proof}

We end this section by showing the Fr\'echet differentiability of the solution to the deterministic equation (\ref{eq:det}), which extends  \cite[Lemma 3 and Proposition 4]{NS}. 


\begin{lemma}\label{lema3}
Assume the hypotheses of Theorem \textnormal{\ref{t:exist}}.
Assume that $b(t,s, \cdot), \sigma(t,s, \cdot)$ belong to $C^3_b$ for all $s,t \in [0,T]$ and that the partial derivatives of $b$ and $\sigma$ satisfy
${\bf (H2)}$ and ${\bf (H1)}$, respectively.
Then the mapping
$$
F:W_2^{1-\alpha}(0,T; \R^m) \times W_1^{\alpha}(0,T; \R^d)
\rightarrow W_1^{\alpha}(0,T; \R^d)
$$
defined by
\begin{equation}
(h,x) \rightarrow F(h,x):=x-x_0-\int_0^{\cdot} b(\cdot,s, x_s) ds-
\int_0^{\cdot} \sigma(\cdot,s,x_s)d(g_s+h_s)
\end{equation}
is Fr\'echet differentiable. Moreover, for any $(h,x)\in W_2^{1-\alpha}(0,T; \R^m) \times W_1^{\alpha}(0,T; \R^d)$, $k \in W_2^{1-\alpha}(0,T; \R^m)$, $v\in W_1^{\alpha}(0,T; \R^d)$, and $i=1,\dots,d$, the Fr\'echet derivatives with respect to $h$ and $x$ are given respectively by
\begin{align} \label{der1}
D_1F(h,x)(k)^i_t&=-\sum_{j=1}^m \int_0^t \sigma^{i,j}(t,s,x_s) dk^j_s, 
\\\label{der2}
D_2F(h,x)(v)^i_t&=v_t^i-\sum_{k=1}^d \int_0^t\partial_{x_k} b^i(t,s,x_s)v_s^k ds- \sum_{k=1}^d \sum_{j=1}^m\int_0^t \partial_{x_k}\sigma^{i,j}(t,s,x_s) 
v_s^k d(g_s^j+h_s^j).
\end{align}
\end{lemma}

\begin{proof}
For $(h,x)$ and $(\tilde{h}, \tilde{x})$ in $W_2^{1-\alpha}(0,T; \R^m) \times W_1^{\alpha}(0,T; \R^d)$ we have
\begin{equation*}
\begin{split}
F(h,x)_t-F(\tilde{h}, \tilde{x})_t&=x_t-\tilde{x}_t-
\int_0^t (b(t,s, x_s)-b(t,s,\tilde{x}_s)) ds \\
&-\int_0^t (\sigma(t,s,x_s)-\sigma(t,s,\tilde{x}_s)) d(g_s+h_s)
-\int_0^t \sigma(t,s,\tilde{x}_s)d(h_s-\tilde{h}_s).
\end{split}
\end{equation*}
Using  \cite[Proposition 2.2(2)]{BR}, we get that
\begin{equation*}
\begin{split}
\bigg\Vert x-\tilde{x}-
\int_0^{\cdot} (b(\cdot,s, x_s)-b(\cdot,s,\tilde{x}_s)) ds \bigg\Vert_{\alpha,1} & \leq 
c_{\alpha, T}\Vert x-\tilde{x} \Vert_{\alpha,1}.
\end{split}
\end{equation*}
From  \cite[Proposition 3.2(2)]{BR}, we obtain
\begin{equation*}
\begin{split}
&\bigg\Vert \int_0^{\cdot} (\sigma(\cdot,s,x_s)-\sigma(\cdot,s,\tilde{x}_s)) d(g_s+h_s) \bigg\Vert_{\alpha,1}\\
 & \qquad \qquad \leq 
c_{\alpha, T} \Vert x-\tilde{x} \Vert_{\alpha,1} \Vert g+h \Vert_{1-\alpha,2}(1+ \Delta(x)+\Delta(\tilde{x})),
\end{split}
\end{equation*}
where
$$
\Delta(x):=\sup_{u \in [0,T]} \int_0^u \frac{\vert x_u-x_s \vert^{\delta}}{(u-s)^{\alpha+1}} ds \leq c_{\alpha, \delta, T} \Vert x \Vert_{1-\alpha}^{\delta}
$$
and similarly $\Delta (\tilde{x})\leq c_{\alpha, \delta, T} \Vert \tilde{x} \Vert_{1-\alpha}^{\delta}$.
Finally, \cite[Proposition 3.2(1)]{BR} yields  to
\begin{equation*}
\begin{split}
\bigg\Vert \int_0^\cdot \sigma(\cdot,s,\tilde{x}_s)d(h_s-\tilde{h}_s) \bigg\Vert_{\alpha,1} & \leq 
c_{\alpha, T} (1+\Vert \tilde{x} \Vert_{\alpha,1} )\Vert h-\tilde{h} \Vert_{1-\alpha,2}.
\end{split}
\end{equation*}
Therefore, $F$ is continuous in both variables $(h,x)$. We  next show the Fr\'echet differentiability. Let $v,w\in W_1^{\alpha}(0,T; \R^d)$. By 
\cite[Proposition 2.2(2) and 3.2(2)]{BR}, we have that
$$
\Vert D_2F(h,x) (v)- D_2F(h,x) (w)\Vert_{\alpha,1} \leq c_{\alpha, T} \Vert v-w \Vert_{\alpha,1}(1+\Vert g+h \Vert_{1-\alpha,2}).
$$
Thus, $D_2F(h,x)$ is a bounded linear operator.
Moreover,
\begin{align*}
&F(h, x+v)_t-F(h,x)_t-D_2F(h,x)(v)_t \\
&\qquad =\int_0^t (b(t,s, x_s)-b(t,s,x_s+v_s)+\partial_x b(t,s,x_s) v_s) ds\\
&\qquad \qquad +\int_0^t (\sigma(t,s,x_s)-\sigma(t,s,x_s+v_s)+\partial_x \sigma(t,s,x_s)v_s) d(g_s+h_s).
\end{align*}
By the mean value theorem and \cite[Proposition 2.2(2)]{BR},
$$
\bigg\Vert\int_0^{\cdot} (b(\cdot,s, x_s)-b(\cdot,s,x_s+v_s)+\partial_x b(\cdot,s,x_s) v_s) ds \bigg\Vert_{\alpha, 1} \leq c_{\alpha, T}  \Vert v \Vert_{\alpha, 1}^2.
$$
Similarly, using \cite[Proposition 3.2(2)]{BR}, we obtain
\begin{equation*} \begin{split}
&\bigg\Vert\int_0^{\cdot} (\sigma(\cdot,s, x_s)-\sigma(\cdot,s,x_s+v_s)+\partial_x \sigma(\cdot,s,x_s) v_s) d(g_s+h_s) \bigg\Vert_{\alpha, 1} \\
&\qquad \qquad \leq c_{\alpha, \delta, T}  \Vert v \Vert_{\alpha, 1}^2 \Vert g+ h\Vert_{1-\alpha,2}.
\end{split}
\end{equation*}
This shows that $D_2F$ is the Fr\'echet derivative with respect to $x$ of $F(h,x)$. Similarly, we show that it is Fr\'echet differentiable with respect to $h$ and the derivative is given by (\ref{der1}).
\end{proof}

\begin{proposition} \label{frechet}
Assume the hypotheses of Lemma \textnormal{\ref{lema3}}.
Then the mapping $$g \in W_2^{1-\alpha}(0,T; \R^m)\rightarrow x(g) \in W_1^{\alpha}(0,T; \R^d)$$ is Fr\'echet differentiable and for any $h\in W_2^{1-\alpha}(0,T; \R^m) $ the derivative in the direction $h$ is given by
\begin{equation*}
D_h x_t^i =\sum_{j=1}^m \int_0^t \Phi_t^{ij}(s) dh_s^j,
\end{equation*}
where for $i=1,\dots,d$, $j=1,\dots,m$, $0\leq s\leq t$
\begin{equation} \label{eqn}
\begin{split}
\Phi_t^{ij}(s)=\sigma^{ij}(t,s,x_s)&+\sum_{k=1}^d \sum_{\ell=1}^m \int_s^t \partial_{x_k} \sigma^{i,\ell}(t,u, x_u) \Phi_u^{kj}(s) dg_u^{\ell} \\
& +\sum_{k=1}^d \int_s^t \partial_{x_k} b^i(t,u, x_u) \Phi_u^{kj}(s) du,
\end{split}
\end{equation}
and $\Phi_t^{ij}(s)=0$ if $s>t$.
\end{proposition}

\begin{proof}
The proof follows similarly as the proof of \cite[Proposition 4]{NS} once we have extended \cite[Proposition 2 and 9]{NS}. We proceed with both  extensions below.
\end{proof}

The next propositions are the extensions of \cite[Proposition 2 and 9]{NS}, respectively.
\begin{proposition} \label{prop2}
Assume the hypotheses of Lemma \textnormal{\ref{lema3}}.
Fix $g \in W^{1-\alpha}_2(0,T; \R^m)$ and consider the linear equation
$$
v_t=w_t+\int_0^t\partial_{x} b(t,s,x_s)v_s ds+\int_0^t \partial_{x}\sigma(t,s,x_s) 
v_s dg_s.
$$
where $w \in C^{1-\alpha}(0,T;\R^d)$. Then there exists a unique solution $v \in  C^{1-\alpha}(0,T;\R^{d })$ such that
\begin{equation} \label{estimate}
\Vert v \Vert_{\alpha,1} \leq c^{(1)}_{\alpha,T}\Vert w \Vert_{\alpha,1} \exp \left(c^{(2)}_{\alpha,T} \Vert g \Vert_{1-\alpha, 2}^{1/(1-2\alpha)}\right),
\end{equation}
for some positive constants $c^{(1)}_{\alpha,T}$ and $c^{(2)}_{\alpha,T}$.
\end{proposition}

\begin{proof}
Existence and uniqueness follows from  \cite{BR} and the estimate (\ref{estimate}) follows from  \cite[Proposition 4.2]{BR} with $\gamma=1$.
\end{proof}

\begin{proposition} \label{prop9}
Assume the hypotheses of Lemma \textnormal{\ref{lema3}}.
Then the solution to the linear equation \textnormal{(\ref{eqn})} is H\"older continuous of order $1-\alpha$ in $t$, uniformly in $s$ and H\"older continuous of order $\beta \wedge (1-\alpha)$ in $s$, uniformly in $t$.
\end{proposition}

\begin{proof}
By the estimates in \cite{BR}, we get
$$
\sup_{s \in [0,T]}\Vert \Phi_{\cdot}(s) \Vert_{1-\alpha}
\leq c_{\alpha, T} (1+ (1+\Vert g \Vert_{1-\alpha, 2}) \sup_{s \in [0,T]}\Vert \Phi_{\cdot}(s) \Vert_{\alpha,1}).
$$
which is bounded by Proposition \ref{prop2}. Therefore, $\Phi_t(s)$ is H\"older continuous of order $1-\alpha$ in $t$, uniformly in $s$. On the other hand, appealing again to Proposition \ref{prop2}, for $s'\leq s \leq t$, we have 
$$
\Vert \Phi_{\cdot}(s)- \Phi_{\cdot}(s')\Vert_{\alpha,1} \leq c^{(1)}_{\alpha,T}\Vert w_{\cdot}(s,s') \Vert_{\alpha,1} \exp \left(c^{(2)}_{\alpha,T} \Vert g \Vert_{1-\alpha, 2}^{1/(1-2\alpha)}\right),
$$
where
\begin{equation*} 
\begin{split}
w_t(s,s')=\sigma(t,s,x_s)-\sigma(t,s',x_{s'})&+\int_{s'}^s \partial_{x} \sigma(t,u, x_u) \Phi_u(s') dg_u\\
& + \int_{s'}^s \partial_{x} b(t,u, x_u) \Phi_u(s') du.
\end{split}
\end{equation*}
We next bound the $\Vert \cdot \Vert_{\alpha,1}$-norm of $w_{\cdot}(s,s')$. For the first term, by the definition of the $\Vert \cdot \Vert_{\alpha,1}$-norm, we have
\begin{equation*} \begin{split}
\Vert \sigma(\cdot,s,x_s)-\sigma(\cdot,s',x_{s'})\Vert_{\alpha,1} &\leq  c_{\alpha, T} \Vert \sigma(\cdot,s,x_s)-\sigma(\cdot,s',x_{s'})\Vert_{1-\alpha}\\
 &\leq  c_{\alpha, \beta,T} (s-s')^{\beta \wedge (1-\alpha)},
\end{split}
\end{equation*}
where we have used \cite[Lemma A.2]{BR} in the last inequality.

For the second term, as $\partial_x \sigma$ is bounded, we obtain
\begin{equation*} \begin{split}
\bigg\Vert \int_{s'}^s \partial_{x} \sigma(\cdot,u, x_u) \Phi_u(s') dg_u\bigg\Vert_{\alpha, 1}&\leq c_{\alpha, T}
\bigg\vert \int_{s'}^s \Phi_u(s') dg_u\bigg  \vert\\
&\leq 
 c_{\alpha,T} (s-s')^{1-\alpha} \Vert g \Vert_{1-\alpha,2} \sup_{s \in [0,T]} \Vert \Phi_{\cdot}(s) \Vert_{\alpha,1},
\end{split}
\end{equation*}
where the last inequality follows from \cite[Proposition 4.1]{N-R}.

Finally, for the last term, as $\partial_xb$ is bounded, we get 
\begin{equation*} \begin{split}
\bigg\Vert \int_{s'}^s\partial_{x} b(t,u, x_u) \Phi_u(s') du\bigg\Vert_{\alpha, 1}&\leq c_{\alpha, T}
\bigg\vert \int_{s'}^s \Phi_u(s') du\bigg  \vert\\
&\leq   c_{\alpha,T} (s-s') \sup_{s \in [0,T]} \Vert \Phi_{\cdot}(s) \Vert_{\alpha,1}.
\end{split}
\end{equation*}

Therefore, we conclude that
$$
\Vert \Phi_{\cdot}(s)- \Phi_{\cdot}(s')\Vert_{\alpha,1} \leq c_{\alpha, \beta,T} (s-s')^{\beta \wedge (1-\alpha)}\exp \left(c^{(2)}_{\alpha,T} \Vert g \Vert_{1-\alpha, 2}^{1/(1-2\alpha)}\right),
$$
which implies that $\Phi_t(s)$ is H\"older continuous of order $\beta \wedge (1-\alpha)$ in $s$ uniformly in $t$.
\end{proof}

\renewcommand{\theequation}{4.\arabic{equation}}
\setcounter{equation}{0}
\section{Stochastic Volterra equations driven by fBm}

In this section we apply the results obtained in Section 3 to the Volterra equation \eqref{eq:Volterra}.
Recall that $W^H=\{ W_t^H,\, t\in [0,T]\}$ is an $m$-dimensional fractional Brownian motion with Hurst parameter $H>\frac12$. That is, a centered Gaussian  process with covariance function
\[\be(W_t^{H,i} W_t^{H,j})=R_H(t,s)=\frac12\left( t^{2H}+s^{2H}-|t-s|^{2H}\right)\delta_{ij}.\]
Fix $\alpha \in (1-H, \frac12)$.
As the trajectories of $W^H$ are $(1-\alpha+\epsilon)$-H\"older continuous for all $\epsilon<H+\alpha-1$,
by the first inclusion in (\ref{inclusion}), we can apply the framework of Section 3. In particular, under the assumptions of Theorem \ref{th:2}, there exists a unique solution to equation \eqref{eq:Volterra} satisfying
\begin{equation*} 
\sup_{0\leq t\leq T} |X_t|\leq (|X_0|+1)\exp\left(2T\left(\left(K_{T,\alpha}^{(3)}+K_{T,\alpha,\beta}^{(4)}\|W^H\|_{1-\alpha}\right)^{1/(1-\alpha)}\vee 1 \vee T\right) \right).
\end{equation*}
Moreover, under the further assumptions of Theorem \ref{th:1}, we have the estimate
\begin{equation*}
\sup_{0\leq t\leq T} |X_t|\leq |X_0|+1+ T\left(\left(K_{T,\alpha}^{(1)}+ K_{T,\alpha,\beta}^{(2)}\|W^H\|_{1-\alpha}\right)^{1/(1-\alpha)}\vee 1 \vee T\right).
\end{equation*}

As a consequence of these estimates we can establish the following integrability properties of the solution to \eqref{eq:Volterra}.
\begin{theorem}
Assume that $\be(|X_0|^p)<\infty$ for all $p\geq 2$ and that $\sigma$ and $b$ satisfy the hypotheses of Theorem \ref{th:2}. 
Then for all $p\geq 2$
\[\be\left(\sup_{0\leq t\leq T} |X_t|^p\right)<\infty.\]
Moreover, if for any $\lambda>0$ and $\gamma<2H$, $\be\left(\exp(\lambda|X_0|^\gamma)\right)<\infty$,  then under the assumptions of Theorem \ref{th:1}, we have
\[\be\left(\exp\left(\lambda\left(\sup_{0\leq t\leq T}|X_t|^\gamma\right)\right)\right)<\infty,\]
for any $\lambda>0$ and $\gamma<2H$.
\end{theorem}

We next proceed with the study of the existence and smoothness of the density of the solution to (\ref{eq:Volterra}). From now on we assume that the initial condition is constant, that is,  $X_0=x_0 \in \R^d$.
We start by extending the results in \cite{NS} in order to show the existence of the density of the solution to the Volterra equation (\ref{eq:Volterra}) when $\sigma$ is not necessarily bounded. 
We first derive the (local) Malliavin differentiability of the solution.
\begin{theorem} \label{d12loc}
Assume the hypotheses of Lemma \ref{lema3}. Then the solution to \eqref{eq:Volterra}
is almost surely differentiable in the directions of the Cameron-Martin space. Moreover, for any $t>0$, $X_t^i$ belongs to the space $\mathbb{D}^{1,2}_{\textnormal{loc}}$ and the derivative satisfies for $i=1,\dots,d$, $j=1,\dots,m$,
\begin{equation} \label{eq:der}\begin{split}
D_s^jX_t^i=\sigma^{ij}(t,s,X_s)&+\sum_{k=1}^d\sum_{\ell=1}^m
\int_s^t \partial_{x_k} \sigma^{i \ell}(t,r,X_r) D_s^jX_r^k dW_r^{H,\ell}\\
&+\sum_{k=1}^d \int_s^t \partial_{x_k} b^i(t,r,X_r)D_s^j X_r^k dr,
\end{split}
\end{equation}
if $s\leq t$ and 0 if $s>t$.
\end{theorem}

\begin{proof}
By Proposition \ref{frechet}, the mapping
$$\omega \in W_2^{1-\alpha}(0,T; \R^m)\rightarrow X(\omega) \in W_1^{\alpha}(0,T; \R^d)$$ is Fr\'echet differentiable and for all $\varphi \in \mathcal{H}$ and $i=1,\ldots,d$, the Fr\'echet derivative
$$
D_{\mathcal{R}_H \varphi}X_t^i=\frac{d}{d\epsilon}X_t^i(\omega+\epsilon\mathcal{R}_H \varphi) \vert_{\epsilon=0}
$$
exists, which proves the first statement of the theorem.
Moreover, by \cite[Proposition 4.1.3.]{Nua}, this implies that for any $t>0$ 
$X_t^i$ belongs to the space $\mathbb{D}^{1,2}_{\text{loc}}$.

The derivative $
D_{\mathcal{R}_H \varphi}X_t^i$  coincides with
$
\langle DX_t^i, \varphi \rangle_{\mathcal{H}},
$
where $D$ is the usual Malliavin derivative. Furthermore, by Proposition \ref{frechet}, for any $\varphi \in \mathcal{H}$ and $i=1,\ldots,d$,
\begin{equation*} \begin{split}
D_{\mathcal{R}_H \varphi} X_t^i &=\sum_{j=1}^m \int_0^t \Phi_t^{ij}(s) d(\mathcal{R}_H \varphi)^j(s)\\
&=\sum_{j=1}^m \int_0^t \Phi_t^{ij}(s) \left( \int_0^s \partial_s K_H(s,u)(\mathcal{K}_H^{\ast} \varphi)^j (u) du\right) ds\\
&=\sum_{j=1}^m \int_0^T (\mathcal{K}_H^{\ast} \Phi_t^i)^j(s) (\mathcal{K}_H^{\ast} \varphi)^j (s) ds\\
&=\langle \Phi_t^i, \varphi \rangle_{\mathcal{H}}
\end{split}
\end{equation*}
and equation (\ref{eq:der}) follows from (\ref{eqn}). This concludes the proof.
\end{proof}

We next derive the existence of the density.
\begin{theorem}
Assume the hypotheses of Lemma \ref{lema3}. Assume also the following nondegeneracy condition on $\sigma$: for all $s,t \in [0,T]$, the vector space spanned by 
$$
\{ (\sigma^{1j}(t,s,x_0), \dots, \sigma^{dj}(t,s,x_0)), 1\leq j \leq m\}
$$
is $\R^d$. Then, for any $t>0$ the law of the random vector $X_t$ is absolutely continuous with respect to the Lebesgue measure on $\R^d$.
\end{theorem}

\begin{proof}
By Theorem \ref{d12loc} and \cite[Theorem 2.1.2]{Nua} it suffices to show that the Malliavin matrix $\Gamma_t$ of $X_t$ defined by
$$
\Gamma_t^{ij}=\langle DX_t^i, DX_t^j \rangle_{\mathcal{H}}
$$
is invertible a.s., which follows along the same lines as in the proof of \cite[Theorem 8]{NS}.
\end{proof}

We finally consider the case that $\sigma$ is bounded and show the existence and smoothness of the density.
As before, we first study the Malliavin differentiability of the solution.
\begin{theorem} \label{ma}
Assume the hypotheses of Theorem \ref{th:1},
that $b^i(t,s, \cdot), \sigma^{i,j}(t,s, \cdot)$ belong to $C^\infty_b$ for all $s,t \in [0,T]$ and that the partial derivatives of all orders of $b$ and $\sigma$ satisfy
${\bf (H2)}$ and ${\bf (H1)}$ respectively. Then  for any $t>0$, $X_t^i$ belongs to the space $\mathbb{D}^{\infty}$ and the $n$th iterated derivative satisfies the following equation
for $i=1,\dots,d$, $j_1,\ldots,j_n \in \{1,\dots,m\}$,
\begin{equation} \label{iter_der}\begin{split}
D_{s_1}^{j_1} \cdots D_{s_n}^{j_n} X_t^i 
=&\sum_{q=1}^n D_{s_1}^{j_1} \cdots  \check{D}_{s_{q}}^{j_{q}} \cdots D_{s_n}^{j_n}\sigma^{ij_{\ell}}(t,s_{\ell},X_{s_{\ell}})\\
&+\sum_{\ell=1}^m
\int_{s_1 \vee \cdots \vee s_n}^t  D_{s_1}^{j_1} \cdots  D_{s_n}^{j_n}\sigma^{i \ell}(t,r,X_r)  dW_r^{H,\ell}\\
&+\int_{s_1 \vee \cdots \vee s_n}^t D_{s_1}^{j_1} \cdots  D_{s_n}^{j_n}b^i(t,r,X_r) dr,
\end{split}
\end{equation}
if $s_1 \vee \cdots \vee s_n\leq t$ and 0 otherwise. The notation $\check{D}_{s_q}^{j_q}$ means that the factor $D_{s_{q}}^{j_q}$ is omitted in the sum. When $n=1$ this equation coincides with \textnormal{(\ref{eq:der})}.
\end{theorem}

\begin{proof}
By Theorem \ref{d12loc}, for any $t>0$ $X_t^i$ belongs to $\bbd^{1,2}_{\textnormal{loc}}$
and the Malliavin derivative satisfies \eqref{eq:der}.
Applying Theorem \ref{th:3} to the system formed by equations \eqref{eq:Volterra} and \eqref{eq:der} we obtain that a.s. 
\begin{equation} \label{dere}
\sup_{s,t \in [0,T]} |D_s^jX_t^i| \leq 2\left(\|\sigma\|_\infty+1\right) \exp\left(T\left(K_{T,\alpha}^{(5)}+K_{T,\alpha,\beta}^{(6)}\|W^H\|_{1-\alpha}\right)^{1/(1-\alpha)} \vee 1 \vee T\right),
\end{equation}
which implies that for all $p\geq 2$,
\[\sup_{t \in [0,T]}\be\left(\left|\sum_{j=1}^m \int_0^t\int_0^t D_s^{j} X_t^i D_r^{j} X_t^i |r-s|^{2H-2}ds dr\right|^p\right)<\infty.\]
This and \cite[Lemma 4.1.2]{Nua} show that the random variable $X_t^i$ belongs to the Sobolev space $\bbd^{1,p}$ for all $p\geq 2$.
Similarly, it can be proved that $X_t^i$ belongs to the Sobolev space $\bbd^{k,p}$ for all $p,\,k\geq 2$. For the sake of conciseness, we only sketch the main steps. First, by induction, following exactly along the same lines as in the proofs of \cite[Proposition 5 and Lemma 10]{NS} and Proposition \ref{frechet}, it can be shown that the deterministic mapping $x$ defined in Section 3 is infinitely differentiable. Second, by a similar argument as in the proof of Theorem \ref{d12loc}, we have that for all $t>0$, $X_t^i$ is almost surely infinitely differentiable in the directions of the Cameron-Martin space and it belongs to the space $\bbd^{k,p}_{\text{loc}}$ for all $p,\,k\geq 2$. Finally, using equation (\ref{iter_der}), the estimate for linear equations obtained in Theorem \ref{th:3} and an induction argument, we obtain that for all $k, p \geq 2$,
$$
\sup_{t \in [0,T]} \be\left(\Vert D^{(k)} X_t \Vert^p_{\mathcal{H}^{\otimes k}}\right)<\infty,
$$
where $D^{(k)}$ denotes the $k$th iterated derivative.
This concludes the desired claim.
\end{proof}
The next theorem extends and corrects the proof of \cite[Theorem 7]{HN} as there is a mistake in the last step of the proof. 
\begin{theorem}
Assume the hypotheses of Theorem \ref{ma} and that $\sigma(t,s,\cdot)$ is uniformly elliptic,  that is, for all $s,t \in [0,T]$, $x, \xi \in \R^d$ with $\vert \xi \vert=1$,
$$
\sum_{j=1}^m \big(\sum_{i=1}^d \sigma^{i j}(t,s,x) \xi_i\big)^2 \geq \rho^2>0,
$$ 
for some $\rho>0$.
Then for any $t>0$ the probability law of $X_t$ has an $C^\infty$ density.
\end{theorem}

\begin{proof}
By
\cite[Theorem 7.2.6]{NuaNua} it suffices to show that ${\rm E} ( (\text{det} (\Gamma_t))^{-p} )<\infty $ for all $p>1$.
We write
$$
\text{det} (\Gamma_t) \geq \inf_{\vert \xi \vert=1} (\xi^T \Gamma_t \xi)^d. 
$$
Fix $\xi \in \R^d$ with $\Vert \xi \Vert=1$ and $\epsilon \in (0,1)$. Then
\begin{align*}
\xi^T \Gamma_t \xi&=\Vert \sum_{i=1}^d D X^i_t \xi_i
 \Vert^2_{\mathcal{H}}=\Vert \sum_{i=1}^d K^{\ast}_H (D X^i_t) \xi_i
 \Vert^2_{L^2(0,t;\R^m)} \\
 &\geq \Vert \sum_{i=1}^d K^{\ast}_H (D X^i_t ) \xi_i
 \Vert^2_{L^2(t-\epsilon,t;\R^m)}\geq \frac12 A-B,
\end{align*}
where 
\begin{align*}
A:=&\sum_{j=1}^m \int_{t-\epsilon}^t  \sum_{i,k=1}^d \int_s^t \int_s^t\sigma^{ij}(t,u, X_u)  \sigma^{kj}(t,v, X_v) \partial_u K_H(u,s) \partial_v K_H(v,s)\xi_i \xi_k du  dv ds,\\
B:=&\sum_{j=1}^m \int_{t-\epsilon}^t \bigg( \sum_{i=1}^d \int_s^t 
\bigg(\sum_{k=1}^d\sum_{\ell=1}^m
\int_u^t \partial_{x_k} \sigma^{i\ell}(t,r,X_r) D_u^jX_r^k dW_r^{H,\ell} \\
&\quad +\sum_{k=1}^d \int_u^t \partial_{x_k} b^i(t,r,X_r)D_u^j X_r^k dr
\bigg)\partial_u K_H(u,s) \xi_i du  \bigg)^2 ds.
\end{align*}
We next we add and substract the term
$\sigma^{ij}(t,u, X_u) \sigma^{kj}(t,u, X_u)$ inside $A$ to obtain
that $A=A_1+A_2$, where
\begin{align*}
A_1:=&\sum_{j=1}^m \int_{t-\epsilon}^t \int_s^t \int_s^t \left(\sum_{i=1}^d \sigma^{i j}(t,u,X_u) \xi_i\right)^2  \partial_u K_H(u,s) \partial_v K_H(v,s) du dv ds,\\
A_2:=&\sum_{j=1}^m \int_{t-\epsilon}^t  \sum_{i,k=1}^d \int_s^t \int_s^t\sigma^{ij}(t,u, X_u)\left(\sigma^{kj}(t,v, X_v) - \sigma^{kj}(t,u, X_u)\right)\\
&\qquad \qquad \times \partial_u K_H(u,s) \partial_v K_H(v,s)\xi_i \xi_k du  dv ds.
\end{align*}
By the uniform ellipticity property, we get that
\begin{align*}
A_1 &\geq \rho^2  \int_{t-\epsilon}^t \left(\int_s^t \partial_u K_H(u,s) du\right)^2 ds = c_H \int_{t-\epsilon}^t \left(\int_s^t (\frac{u}{s})^{H-1/2}(u-s)^{H-\frac32} du\right)^2 ds\\
& \geq c_H \int_{t-\epsilon}^t \left(\int_s^t (u-s)^{H-\frac32} du\right)^2 ds=c_H \epsilon^{2H}.
\end{align*}
Moreover, since $\sigma$ is bounded, using H\"older's inequality and hypothesis {\bf (H1)},
for any $q\geq 1$, we get that
\begin{align*}
{\rm E} \vert A_2 \vert^q &\leq C_{T,q} \epsilon^{3(q-1)} \int_{t-\epsilon}^t   \int_{t-\epsilon}^t \int_{t-\epsilon}^t\big({\rm E}(\vert X_u-X_v\vert^q)+{\rm E}(\vert X_u \vert^q )\vert u-v\vert^{\beta q}\\
&\qquad\qquad  + {\rm E}(\vert X_u \vert^q \vert X_u-X_v\vert^{\delta q})\big)(\partial_u K_H(u,s) \partial_v K_H(v,s))^q du  dv ds \\
&\leq C_{T,q} \epsilon^{3(q-1)} \int_{t-\epsilon}^t   \int_{t-\epsilon}^t \int_{t-\epsilon}^t\left(\vert u-v\vert^{(1-\alpha) q}+\vert u-v\vert^{\beta q}+\vert u-v\vert^{\delta q (1-\alpha)} \right)\\
&\qquad \qquad \times (\partial_u K_H(u,s) \partial_v K_H(v,s))^q du  dv ds \\
&\leq C_{T,q} \epsilon^{q(2H+\min\{1-\alpha, \beta, \delta(1-\alpha)\}}.
\end{align*}
We are left to bound ${\rm E} \vert B \vert^q$. Since $\partial_x b(t,s,x)$ is bounded, using H\"older's inequality and (\ref{dere}), we obtain that for all $q \geq 1$,
\begin{equation*} \begin{split}
&{\rm E} \bigg\vert\sum_{j=1}^m \int_{t-\epsilon}^t \bigg( \sum_{i,k=1}^d \int_s^t 
 \int_u^t \partial_{x_k} b^i(t,r,X_r)D_u^j X_r^k 
\partial_u K_H(u,s) \xi_i dr du  \bigg)^2 ds\bigg\vert^q \\
&\leq C_{\alpha, T,q} \epsilon^{5q-2} \int_{t-\epsilon}^t 
\int_{t-\epsilon}^t (\partial_u K_H(u,s))^{2q} du ds \\
 &\leq C_{\alpha, T,q} \epsilon^{q(2H+2)}.
\end{split}
\end{equation*}
Similarly, for all $q \geq 1$, we have that
\begin{equation*} \begin{split}
&{\rm E} \bigg\vert\sum_{j=1}^m \int_{t-\epsilon}^t \bigg( \sum_{i,k=1}^d \sum_{\ell=1}^m \int_s^t 
\left(\int_u^t \partial_{x_k} \sigma^{i\ell}(t,r,X_r) D_u^jX_r^k dW_r^{H,\ell}\right)
\partial_u K_H(u,s) \xi_i du  \bigg)^2 ds\bigg\vert^q \\
&\leq C_{\alpha, T,q} \epsilon^{3q-2+(1-\alpha)2q} \int_{t-\epsilon}^t 
\int_{t-\epsilon}^t (\partial_u K_H(u,s))^{2q} du ds \\
 &\leq C_{\alpha, T,q} \epsilon^{q(2H+2(1-\alpha))}.
\end{split}
\end{equation*}
Appealing to \cite[Proposition 3.5]{DKN} we conclude the desired result.
\end{proof}

\end{document}